\documentclass[a4paper,fleqn]{cas-sc}

\usepackage[square,numbers,sort&compress]{natbib}

\usepackage{graphicx}          
\usepackage{lineno,hyperref}
\usepackage{amsfonts}
\usepackage{amsmath}
\usepackage{amssymb}
\usepackage{amscd,multirow}
\usepackage{epstopdf}
\usepackage[titletoc]{appendix}
\usepackage{subfigure}
\usepackage[justification=centering]{caption}
\usepackage{url}
\usepackage{cleveref}
\usepackage{tikz}
\usepackage{pifont}
\usepackage{enumitem}

\usepackage{subeqnarray,cases}
\usepackage[ruled]{algorithm2e}
\numberwithin{equation}{section}

\allowdisplaybreaks[4]
\newtheorem{definition}{Definition}
\newtheorem{theorem}{Theorem}

\newtheorem{lemma}{Lemma}

\newtheorem{remark}{Remark}

\newenvironment{proof}{\noindent {\textbf{Proof.}}}

\usepackage{fancyhdr} 

\begin{document}
	\let\WriteBookmarks\relax
	\def\floatpagepagefraction{1}
	\def\textpagefraction{.001}
	\let\printorcid\relax

	\shorttitle{Fast convergence of dynamical systems with implicit Hessian damping and Tikhonov regularization} 
	
	\shortauthors{H.L. Li, X. He, Y.B. Xiao}  
	
	\title[mode = title]{Fast convergence of primal-dual dynamical systems with implicit Hessian damping and Tikhonov regularization}

	\author[1]{Hong-lu Li}\ead{lihonglu_7988@163.com}
	\author[2]{Xin He}\ead{hexinuser@163.com}\cormark[1]
	\author[1]{Yi-bin Xiao}\ead{xiaoyb9999@hotmail.com} 
	\cortext[1]{Corresponding author}

	\address[1]{Department of Mathematics, University of Electronic Science and Technology of China, Chengdu, Sichuan, P.R. China}
	
		\address[2]{School of Science, Xihua University, Chengdu, Sichuan, P.R. China}
	\nonumnote{}
	
	\begin{abstract}
		This paper proposes novel primal-dual dynamical systems for solving linear equality constrained convex optimization. First, we introduce a primal-dual dynamical system with implicit Hessian damping, which can neutralize the transversal oscillations without requiring computation of the Hessian matrix. We establish the fast convergence properties of the proposed dynamical system under suitable conditions. Furthermore, we incorporate a Tikhonov regularization term and prove that the resulting trajectories converge strongly to the minimum norm solution. Numerical experiments are conducted to validate the theoretical findings.
	\end{abstract}

	\begin{keywords}
		Linear equality constrained convex optimization\sep Implicit Hessian damping\sep Tikhonov regularization\sep Convergence properties
	\end{keywords}

	\maketitle

	\section{Introduction}
	\subsection{Problem statement}
Let $\mathcal{X}$ and $\mathcal{Y}$ be two real Hilbert spaces with the inner product $\langle\cdot,\cdot\rangle$ and its induced norm $\Vert \cdot \Vert$. In this paper, we consider the linear equality constrained optimization problem:
\begin{equation}\label{equ1}
	\begin{aligned}
		&\min_{x\in\mathcal{X}}\;\;f(x),\\
		&\;{\rm{s.t.}} \;Ax=b,
	\end{aligned}
\end{equation}
where $f:\mathcal{X}\rightarrow\mathbb{R}$ is a continuously differentiable convex function,  $A:\mathcal{X}\rightarrow\mathcal{Y}$ is a continuous linear operator, and $b\in\mathcal{Y}$. The set of feasible points of the problem \eqref{equ1} is denoted by $\mathbb{F}$, i.e. $\mathbb{F}:=\left\{x\in\mathcal{X}: Ax=b\right\}$.
The problem \eqref{equ1} has attracted significant research interest in recent years due to its wide applicability in various scientific and engineering domains. The problem \eqref{equ1} arises naturally in image recovery \cite{bib28}, machine learning \cite{bib26,bib27}, global consensus problems \cite{bib6,bib7}, compressive sensing \cite{bib8,bib9}, and transportation systems \cite{bib53,bib54}.

Consider the saddle point problem associated with the linear equality constrained convex optimization problem \eqref{equ1}:
\begin{equation*}
	\min_{x\in \mathcal{X}}\max_{\lambda\in\mathcal{Y}}\mathcal{L}(x,\lambda),
\end{equation*}
where the Lagrangian function  $\mathcal{L}: \mathcal{X}\times\mathcal{Y}\rightarrow\mathbb{R}$  is defined by
\[\mathcal{L}(x, \lambda): =f(x)+\langle \lambda, Ax-b\rangle.\]
We denote by $\Omega$ the set of saddle points $\mathcal{L}$, i.e., $(x^*,\lambda^*)\in \Omega$ if and only if for any $(x,\lambda)\in\mathcal{X}\times\mathcal{Y}$, 
\[\mathcal{L}(x^*,\lambda)\leq\mathcal{L}(x^*,\lambda^*)\leq\mathcal{L}(x,\lambda^*).\]
In this case, $x^*\in\mathcal{X}$ is a solution of the problem \eqref{equ1} and $\lambda^*\in\mathcal{Y}$ is a solution of its Lagrange dual. It is well known that the original Lagrangian $\mathcal{L}$ and the augmented Lagrangian $\mathcal{L}_\rho : \mathcal{X} \times \mathcal{Y} \to \mathbb{R}$ 
\begin{equation*}
	\mathcal{L}_{\rho}(x, \lambda): =\mathcal{L}(x, \lambda)+\frac{\rho}{2}\Vert Ax-b\Vert^2.
\end{equation*}
share the same set of saddle points.

In this work, we propose two novel dynamical systems for solving the problem \eqref{equ1}, based on the augmented Lagrangian function $\mathcal{L}_{\rho}$, implicit Hessian damping, and Tikhonov regularization. Moreover, we analyze the convergence properties of each system separately. Throughout the paper, we assume that $\Omega \neq \emptyset$.

\subsection{Related works}
The connection between continuous-time dynamical systems and discrete algorithms for convex optimization has attracted considerable research attention. A landmark contribution in this direction is the work of Su et al. \cite{bib1}, who proposed the following inertial dynamical system with vanishing damping term $\frac{\alpha}{t}$:
\begin{equation*}
	(\mathrm{AVD}_{\alpha}) \quad \ddot{x}(t) + \frac{\alpha}{t} \dot{x}(t) + \nabla f(x(t)) = 0,
\end{equation*}
for solving the unconstrained optimization problem $\min_x f(x)$. When $\alpha = 3$, the dynamical system  $(\mathrm{AVD}_{\alpha})$ can be viewed as a continuous-time limit of Nesterov’s accelerated gradient method \cite{bib2} and the fast iterative shrinkage-thresholding algorithm \cite{bib3}. Su et al. \cite{bib1} established a convergence rate of $f(x(t)) - \min_x f = \mathcal{O}(\frac{1}{t^2})$ when $\alpha \ge 3$. May \cite{bib5} and Attouch et al. \cite{bib4} further analyzed the convergence behavior of the dynamical system $(\mathrm{AVD}_{\alpha})$ in the cases $\alpha > 3$ and $\alpha \leq 3$, respectively. Moreover, Attouch et al. \cite{bib10} extended the study to a generalized version of the dynamical system $(\mathrm{AVD}_{\alpha})$ with a time-dependent scaling factor $\gamma(t)$:
\begin{equation*}
	(\mathrm{AVD_{\alpha,\gamma}})\quad\ddot{x}(t)+\frac{\alpha}{t}\dot{x}(t)+\gamma(t)\nabla f(x(t))=0,
\end{equation*}
and proved the convergence rate $\mathcal{O}(\frac{1}{t^2\gamma(t)})$ of the objective residual. This indicates that the scaling coefficient $\gamma(t)$ may play a crucial role in accelerating the convergence rate.  

It is well known that for the linear equality constrained convex optimization problem \eqref{equ1}, one can typically introduce the Lagrangian function to solve the problem. Based on this, Zeng et al. \cite{bib13} proposed the following inertial primal-dual dynamical system (IPDD) with damping $\frac{\alpha}{t}$:
\begin{equation*}
	\begin{cases}
		\ddot{x}(t)+\frac{\alpha}{t}\dot{x}(t)=-\nabla f(x(t))-A^T(\lambda(t)+\beta t\dot{\lambda}(t))-A^T(Ax(t)-b),\\
		\ddot{\lambda}(t)+\frac{\alpha}{t}\dot{\lambda}(t)=A(x(t)+\beta t\dot{x}(t))-b,
	\end{cases}
\end{equation*}
and proved that under suitable parameter choices, the objective residual and constraint violation decay at rates $\mathcal{O}(\frac{1}{t^2})$ and $\mathcal{O}(\frac{1}{t})$, respectively. Based on the aforementioned dynamical system,  Bo\c{t} et al. \cite{bib11}  and He et al. \cite{bib9} explored its discrete algorithms, and He et al. \cite{bib12} further considered a perturbed vision. Further extensions were proposed by Attouch et al. \cite{bib15} and He et al. \cite{bib14}, who developed new primal-dual dynamical systems with general damping parameters for separable convex optimization with linear constraints. However, recognizing that second-order dynamical systems are more complex than first-order dynamical systems both from a numerical computation perspective and in terms of time discretization, He et al.\cite{HeTAC2022,bib17} proposed a class of ``second-order + first-order'' dynamical systems and analyzed the fast convergence properties. 

\subsubsection{The role of the Hessian Damping}
However, inertial dynamical systems with only viscous damping may exhibit severe oscillations in practice. To address this issue, Hessian-driven damping has been introduced and extensively studied in the literature. These dynamical systems incorporate an additional damping term of the form $\nabla^2 f(x(t)) \dot{x}(t)$, which improves the control of oscillations by accounting for the local curvature of the objective function.  By incorporating Hessian-driven damping into $(\mathrm{AVD_{\alpha}})$, Attouch et al. \cite{bib18} proposed the following inertial dynamical system with Hessian-driven damping:
\begin{equation*}
	(\mathrm{DIN\text{-}AVD})\quad \ddot{x}(t)+\frac{\alpha}{t}\dot{x}(t)+\beta \nabla^2f(x(t))\dot{x}(t)+\nabla f(x(t))=0.
\end{equation*}
They showed that when $\alpha \geq 3$ and $\beta >0$, the objective residual decays at a rate of $\mathcal{O}(\frac{1}{t^2})$.  Furthermore, Attouch et al. \cite{bib51} studied an inertial dynamical system derived from $(\mathrm{DIN\text{-}AVD})$ by incorporating a general scaling coefficient $b(t)$ and established a convergence rate of $f(x(t))- \min f =\mathcal{O}(\frac{1}{t^2b(t)})$. It is worth noting that the dynamical system $(\mathrm{DIN\text{-}AVD})$ requires the objective function to be twice continuously differentiable. Moreover, it has been observed that $(\mathrm{DIN\text{-}AVD})$ does not naturally lead to Nesterov-type inertial algorithms when discretized using standard implicit or explicit schemes. To address this limitation, Alecsa et al. \cite{bib19} proposed an inertial dynamical system with implicit Hessian damping, formulated as follows:
\begin{equation}\label{equ101}
	\ddot{x}(t)+\frac{\alpha}{t}\dot{x}(t)+\nabla f\left(x(t)+\left(\gamma+\frac{\beta}{t}\right)\dot{x}(t)\right)=0.
\end{equation}
They demonstrated that this dynamical system effectively reduces oscillations even when the objective function is not twice continuously differentiable. Specifically, they showed that inertial optimization algorithms, such as Nesterov’s accelerated gradient method, can be naturally derived from their dynamical system through explicit discretization. 

For the linear equality constrained optimization problem \eqref{equ1}, He et al. \cite{bib20} initially proposed a mixed primal-dual dynamical system with explicit Hessian-driven damping. They demonstrated that, under general Hessian-driven damping and scaling coefficients, the proposed dynamical system guarantees fast convergence of the primal-dual gap, the objective residual, and the feasibility violation. For additional research on dynamical systems with Hessian-driven damping, we refer the reader to \cite{bib29,bib30,bib50}. However, it remains unclear whether one can design a second-order primal-dual dynamical system with implicit Hessian-driven damping that avoids explicit computation of the Hessian matrix, while still effectively reducing oscillations in the trajectory.

\subsubsection{The role of the Tikhonov regularization}

As indicated by existing literature, achieving strong convergence of the trajectory generated by a dynamical system is generally difficult without imposing additional assumptions on the objective function, such as strong convexity. To address this challenge and ensure strong convergence of the trajectory even when the objective function is merely convex, researchers have employed Tikhonov regularization techniques. By incorporating a Tikhonov regularization term into the dynamical system, they have been able to establish strong convergence of the trajectory. Attouch et al. \cite{bib22} were the first to propose a Tikhonov regularized dynamical system associated with the heavy-ball system with friction. They proved that if the Tikhonov regularization parameter $\epsilon(t)$ satisfies the condition $\int_{0}^{+\infty} \epsilon(t)\,dt = +\infty$, then the trajectory generated by the dynamical system converges strongly to the minimum-norm solution of the unconstrained optimization problem $\min_{x} f(x)$. To simultaneously achieve rapid convergence in objective function value and strong convergence of the trajectory, Attouch et al. \cite{bib21} proposed a Tikhonov regularized version of the dynamical system $(\mathrm{AVD_{\alpha}})$, given by:
\begin{equation*}
	(\mathrm{AVD_{\alpha,\epsilon}})\quad \ddot{x}(t) + \frac{\alpha}{t} \dot{x}(t) + \nabla f(x(t)) + \epsilon(t) x(t) = 0.
\end{equation*}
They showed that this dynamical system not only retains the fast convergence rate of $(\mathrm{AVD_{\alpha}})$ but also guarantees strong convergence of the trajectory.  Subsequently, Bo\c{t} et al. \cite{bib40} proposed the following dynamical system:
\begin{equation*}
	\ddot{x}(t)+\frac{\alpha}{t}\dot{x}(t)+\beta \nabla^2f(x(t))\dot{x}(t)+\nabla f(x(t))+\epsilon(t) x(t)=0.
\end{equation*}
They prove that the dynamical system can simultaneously inherit the favorable properties of both the dynamical system $(\mathrm{DIN\text{-}AVD})$ and the dynamical system $(\mathrm{AVD_{\alpha,\epsilon}})$. In addition, Alecsa et al. \cite{bib37} studied a Tikhonov regularized inertial dynamical system with implicit Hessian damping. They proved that the objective function value along the generated trajectory converges to the global minimum at a rate of $o(\frac{1}{t^2})$ and that the trajectory converges strongly.

More recently, Zhu et al. \cite{bib23} proposed a Tikhonov regularized primal-dual dynamical system for the linear equality constrained optimization problem \eqref{equ1}:
\begin{equation*}
	\begin{cases}
		\ddot{x}(t)+\frac{\alpha}{t}\dot{x}(t)=-\beta(t)(\nabla f(x(t))+A^T(\lambda(t)+\theta t\dot{\lambda}(t))+\rho A^T(Ax(t)-b)+\epsilon(t)x(t)),\\
		\ddot{\lambda}(t)+\frac{\alpha}{t}\dot{\lambda}(t)=\beta(t)(A(x(t)+\theta t\dot{x}(t))-b),
	\end{cases}
\end{equation*}
They showed that this dynamical system achieves the convergence rate of $\mathcal{O}(\frac{1}{t^2\beta(t)})$ for both the primal–dual gap and the feasibility violation, when the Tikhonov regularization parameter satisfies $\int_{t_0}^{+\infty} t \beta(t) \epsilon(t) \,dt < +\infty$. Furthermore, under the condition $\int_{t_0}^{+\infty} \frac{\beta(t)\epsilon(t)}{t} \,dt < +\infty$, the trajectory converges strongly to the minimum norm solution of the problem \eqref{equ1}. For additional research on Tikhonov regularized dynamical systems, we refer the reader to \cite{bib18-1,bib25,bib36,bib38,bib40}.

\subsection{Main contributions}

In this paper, we propose two novel primal-dual dynamical systems with
implicit Hessian damping and Tikhonov regularization to solve the linear equality constrained optimization problem \eqref{equ1}. Our main contributions are summarized as follows.
\begin{itemize}
	\item [(a)] We first propose a primal-dual dynamical system with implicit Hessian damping for solving the constrained optimization problem \eqref{equ1}, and establish its fast convergence properties. Our results can be regarded as an extension of the dynamical system \eqref{equ101} proposed in \cite{bib19} to constrained optimization problems. Moreover, in contrast to dynamical systems with explicit Hessian-driven damping \cite{bib18,bib51,bib20}, which have been developed for either unconstrained problems or the linearly constrained problem \eqref{equ1}, the proposed dynamical system effectively reduces oscillations and broadens applicability by not requiring the objective function to be twice differentiable.
	
	\item [(b)] To further ensure the strong convergence of the trajectory, we incorporate a Tikhonov regularization term $\epsilon(t)x(t)$ into the proposed primal-dual dynamical system with implicit Hessian damping, resulting in a new primal-dual dynamical system with implicit Hessian damping and Tikhonov regularization. We rigorously prove that the resulting dynamical system \eqref{equ2_1} not only inherits the desirable fast convergence and reduced oscillatory behavior but also guarantees the strong convergence of the trajectory. The dynamical system can be viewed as an extension of the dynamical system proposed in \cite{bib37} to constrained optimization problems.
	
\end{itemize}

\subsection{Organization}
In this paper, Section \ref{sec2} focuses on analyzing the asymptotic convergence properties of the primal-dual dynamical system with implicit Hessian damping. In Section \ref{sec3}, we analyze the convergence properties of the primal-dual dynamical system with implicit Hessian damping and Tikhonov regularization and establish the strong convergence of its trajectory. Section \ref{sec5} includes several numerical experiments designed to illustrate the theoretical results obtained in the previous sections. Finally, Section \ref{sec6} provides a conclusion that highlights the key contributions of this research.

\section{Primal-dual dynamical system with implicit Hessian damping}\label{sec2}

In this section, inspired by the use of implicit Hessian damping as in \cite{bib19} and scaling coefficients,  we propose the following ``second-order primal'' + ``second-order dual'' dynamical system for solving  \eqref{equ1}:
\begin{equation}\label{equ2}
	\begin{cases}
		\begin{aligned}
			&\ddot{x}(t) + \frac{\alpha}{t} \dot{x}(t) + \xi(t) \nabla_x\mathcal{L}_{\rho}(\hat{x}(t),\bar{\lambda}(t)) = 0, \\
			&\ddot{\lambda}(t) + \frac{\alpha}{t} \dot{\lambda}(t) - \xi(t) (\nabla_{\lambda}\mathcal{L}_{\rho}(\bar{x}(t),\lambda(t))+\eta(t)A\dot{x}(t)
			-\frac{3}{2\alpha} t\xi(t)\beta(t)A\nabla_x\mathcal{L}_{\rho}(\hat{x}(t), \bar{\lambda}(t)))= 0,
		\end{aligned}
	\end{cases} 
\end{equation}
where $\alpha>0$, $t\ge t_0>0$, 
\[ \hat{x}(t)=x(t)+\beta(t)\dot{x}(t) \text{ with } \beta(t)=\gamma+\frac{\beta}{t}, \text{ and } \gamma>0,~ \beta\in\mathbb{R} \text{ or } \gamma=0,~ \beta\ge 0, \]
\begin{equation*}\label{eq_coff}
	(\bar{x}(t),\bar{\lambda}(t)) =(x(t),\lambda(t))+\frac{3}{2\alpha}t(\dot{x}(t), \dot{\lambda}(t)),
\end{equation*}
$\eta(t)=\frac{1}{2\alpha}(-\alpha\beta(t)+3 t\dot{\beta}(t))$ and $\xi:\left[t_0,+\infty\right)\rightarrow\left(0,+\infty\right)$ is a continuous function.

The motivation of the dynamical system \eqref{equ2} in relation to linearly constrained optimization problem \eqref{equ1} is motivated by the following facts.

\textbf{(a)} The term $\hat{x}(t) = x(t) + \beta(t)\dot{x}(t)$ can be interpreted as the primal trajectory $x(t)$ perturbed by $\beta(t)\dot{x}(t)$. Consider the expression $\nabla_x\mathcal{L}_{\rho}(\hat{x}(t), \bar{\lambda}(t)) = \nabla f(\hat{x}(t)) + A^T \bar{\lambda}(t) + \rho A^T(A\hat{x}(t) - b)$. By applying a Taylor expansion of $\nabla f$, we obtain
\[ \nabla f(x(t) + \beta(t)\dot{x}(t))\approx \nabla f(x(t))+\beta(t)\nabla^2f(x(t))\dot{x}(t).\]
We observe that the term $\nabla_x\mathcal{L}_{\rho}(\hat{x}(t), \bar{\lambda}(t))$ can be regarded as an implicit Hessian damping. This leads to an implicit variant of the primal-dual dynamical system with Hessian damping, similar to that in \cite{bib20}. We will demonstrate that even without explicitly computing the Hessian $\nabla^2 f$, our dynamical system can still suppress transversal oscillations, as effectively as the explicit Hessian damping dynamical system proposed in \cite{bib20}. When $\beta(t)=0$, it reduces to the inertial primal-dual dynamical system in \cite{bib16,bib13}. In the special case where $A = 0$ and $b = 0$, the primal part of the dynamical system \eqref{equ2} reduces to the second-order dynamical system studied in \cite{bib19}.

\textbf{(b)} The vanishing damping term $\frac{\alpha}{t}$ plays a critical role in achieving Nesterov-type acceleration; see, e.g., \cite{bib1,bib10,bib13,bib20,bib17} for related studies. The time-scaling function $\xi(t)$ contributes to achieving further acceleration, as demonstrated in several studies \cite{bib15,bib25,bib17}. The extrapolation terms $\frac{3}{2\alpha}t(\dot{x}(t), \dot{\lambda}(t))$ and $\eta(t)A\dot{x}(t)$ are commonly used in inertial primal-dual dynamical systems to solve the problem \eqref{equ1}, and are essential for achieving fast convergence of the dynamical system; see, e.g., \cite{bib13,bib11,bib15,bib23}.

The existence and uniqueness of a global solution to the dynamical system \eqref{equ2} can be readily established using the Cauchy-Lipschitz-Picard theorem \cite[Proposition 6.2.1]{bib24}. For completeness, a detailed proof is provided in Appendix \ref{Appendix A} with the Tikhonov regularization parameter $\epsilon$ set to zero in the present case. We then proceed to establish several convergence results for the dynamical system \eqref{equ2}. To start with,  let $(x(t), \lambda(t))$ be a solution of the dynamical system $\rm{\eqref{equ2}}$ and $(x^*, \mu)\in\mathbb{F}\times\mathcal{Y}$. We first define an energy function $E_{\mu}: [t_0, +\infty)\rightarrow \mathbb{R}$ as follows:
\begin{equation}\label{equ3}
	\begin{aligned}
		E_{\mu}(t)=&t^2\xi(t)(\mathcal{L}_{\rho}(x(t)+\beta(t)\dot{x}(t), \mu)-\mathcal{L}_{\rho}(x^*, \mu))\\
		&+\frac{1}{2}\Vert \frac{2\alpha}{3}(x(t)-x^*)+t\dot{x}(t)\Vert^2+\frac{\alpha(\alpha-3)}{9}\Vert x(t)-x^*\Vert^2\\
		&+\frac{1}{2}\Vert \frac{2\alpha}{3}(\lambda(t)-\mu)+t\dot{\lambda}(t)\Vert^2+\frac{\alpha(\alpha-3)}{9}\Vert \lambda(t)-\mu\Vert^2.
	\end{aligned}
\end{equation}
In what follows, we provide an estimate for its derivative. This estimate will be crucial in analyzing the asymptotic convergence properties of the dynamical system \eqref{equ2}.

\begin{lemma}\label{lem2.1}
	Let $t_0>0$, $(x(t), \lambda(t))$ be a solution of the dynamical system $\rm{\eqref{equ2}}$ and $(x^*, \mu)\in\mathbb{F}\times\mathcal{Y}$. Denote the energy function $E_{\mu}$ as in \eqref{equ3}. Assume that $\alpha>3$ and $\xi$ is a differentiable function that satisfies the condition: there exist $t_1\ge t_0$ and a sufficiently small constant $\delta>0$ such that 
	\begin{equation}\label{equ15}
		\left(2-\frac{2\alpha}{3}\right)+t\frac{\dot{\xi}(t)}{\xi(t)}< -\delta,\;(\forall t\ge t_1).
	\end{equation} Then, there exist $t_2\ge t_0$ and $p_2>0$  such that for all $t\ge t_2$, 
	\begin{equation}\label{equ17}
		\begin{aligned}
			\dot{E}_{\mu}(t)&+\delta t\xi(t)(\mathcal{L}_{\rho}(x(t)+\beta(t)\dot{x}(t), \mu)-\mathcal{L}_{\rho}(x^*, \mu))
			+\frac{\rho\alpha}{3}t\xi(t)\Vert A(x(t)+\beta(t)\dot{x}(t))-b\Vert^2\\
			&
			+p_2t\Vert \dot{x}(t)\Vert^2+ \frac{\alpha-3}{3}t\Vert \dot{\lambda}(t)\Vert^2\leq 0.
		\end{aligned}
	\end{equation}
\end{lemma}
\begin{proof}
	Differentiating \eqref{equ3} with respect to $t$ yields
	\begin{equation}\label{equ4}
		\begin{aligned}
			\dot{E}_{\mu}(t)=&(2t\xi(t)+t^2\dot{\xi}(t))(\mathcal{L}_{\rho}(x(t)+\beta(t)\dot{x}(t), \mu)-\mathcal{L}_{\rho}(x^*,  \mu))\\
			&+t^2\xi(t)\langle \nabla_x\mathcal{L}_{\rho}(x(t)+\beta(t)\dot{x}(t), \mu), (1+\dot{\beta}(t))\dot{x}(t)+\beta(t)\ddot{x}(t)\rangle\\
			&+\frac{4\alpha^2}{9}\langle x(t)-x^*+\frac{3}{2\alpha} t\dot{x}(t),(1+\frac{3}{2\alpha} )\dot{x}(t)+\frac{3}{2\alpha}  t\ddot{x}(t)\rangle +\frac{2\alpha(\alpha-3)}{9}\langle x(t)-x^*,\dot{x}(t)\rangle\\
			&+\frac{4\alpha^2}{9}\langle \lambda(t)-\mu+\frac{3}{2\alpha}  t\dot{\lambda}(t), (1+\frac{3}{2\alpha} )\dot{\lambda}(t)+\frac{3}{2\alpha}  t\ddot{\lambda}(t)\rangle +\frac{2\alpha(\alpha-3)}{9}\langle \lambda(t)-\mu,\dot{\lambda}(t)\rangle.
		\end{aligned}
	\end{equation}
	Meanwhile, from the definition of the dynamical system \eqref{equ2}, it follows that
	\begin{equation*}
		\begin{aligned}
			&t^2\xi(t)\langle \nabla_x\mathcal{L}_{\rho}(x(t)+\beta(t)\dot{x}(t),\mu), (1+\dot{\beta}(t))\dot{x}(t)+\beta(t)\ddot{x}(t)\rangle\\
			=&t^2\xi(t)(1+\dot{\beta}(t)-\frac{\alpha}{t}\beta(t))\langle \nabla_x\mathcal{L}_{\rho}(x(t)+\beta(t)\dot{x}(t),\lambda(t)+\frac{3}{2\alpha} t\dot{\lambda}(t)),\dot{x}(t)\rangle\\
			&+t^2\xi(t)(1+\dot{\beta}(t)-\frac{\alpha}{t}\beta(t))\langle \mu-(\lambda(t)+\frac{3}{2\alpha} t\dot{\lambda}(t)), A\dot{x}(t)\rangle\\
			&-t^2\xi(t)^2\beta(t)\Vert \nabla_x\mathcal{L}_{\rho}(x(t)+\beta(t)\dot{x}(t), \lambda(t)+\frac{3}{2\alpha} t\dot{\lambda}(t))\Vert ^2\\
			&-t^2\xi(t)^2\beta(t)\langle A^T(\mu-(\lambda(t)+\frac{3}{2\alpha} t\dot{\lambda}(t))), \nabla_x\mathcal{L}_{\rho}(x(t)+\beta(t)\dot{x}(t),\lambda(t)+\frac{3}{2\alpha} t\dot{\lambda}(t))\rangle,
		\end{aligned}
	\end{equation*}
	and
	\begin{equation*}
		\begin{aligned}
			&\frac{4\alpha^2}{9}\langle x(t)-x^* +\frac{3}{2\alpha} t\dot{x}(t),(1+\frac{3}{2\alpha} )\dot{x}(t)+\frac{3}{2\alpha}  t\ddot{x}(t)\rangle \\
			=&\frac{2\alpha(3-\alpha)}{9}\langle x(t)-x^*, \dot{x}(t)\rangle+ \frac{3-\alpha}{3}t\Vert \dot{x}(t)\Vert^2-\frac{2\alpha}{3}t\xi(t)\langle A^T(\lambda(t)+\frac{3}{2\alpha} t\dot{\lambda}(t)-\mu), x(t)-x^*\rangle\\
			&-\frac{2\alpha}{3}t\xi(t)\langle \nabla_x\mathcal{L}_{\rho}(x(t)+\beta(t)\dot{x}(t), \mu), x(t)-x^*\rangle 
			-t^2\xi(t)\langle \nabla_x\mathcal{L}_{\rho}(x(t)+\beta(t)\dot{x}(t), \lambda(t)+\frac{3}{2\alpha} t\dot{\lambda}(t)), \dot{x}(t) \rangle.
		\end{aligned}
	\end{equation*}
	Furthermore, together with the definition of augmented Lagrange function, we have
	\begin{equation*}
		\begin{aligned}
			& \frac{4\alpha^2}{9}\langle \lambda(t)-\mu+\frac{3}{2\alpha}  t\dot{\lambda}(t), (1+\frac{3}{2\alpha} )\dot{\lambda}(t)+\frac{3}{2\alpha}  t\ddot{\lambda}(t)\rangle \\
			=& \frac{2\alpha(3-\alpha)}{9}\langle \lambda(t)-\mu, \dot{\lambda}(t)\rangle+ \frac{3-\alpha}{3}t\Vert \dot{\lambda}(t)\Vert^2
			+\frac{2\alpha}{3}t\xi(t)\langle A(x(t)+\frac{3}{2\alpha} t\dot{x}(t))-b, \lambda(t)-\mu\rangle\\
			&+t^2\xi(t)\langle A(x(t)+\frac{3}{2\alpha} t\dot{x}(t))-b, \dot{\lambda}(t) \rangle 
			+ \frac{-\alpha\beta(t)+3 t\dot{\beta}(t)}{3}t\xi(t)\langle \lambda(t)-\mu+\frac{3}{2\alpha} t\dot{\lambda}(t), A\dot{x}(t)\rangle\\
			&- t^2\xi(t)^2\beta(t)\langle \lambda(t)-\mu+\frac{3}{2\alpha} t\dot{\lambda}(t), A\nabla_x\mathcal{L}_{\rho}(x(t)+\beta(t)\dot{x}(t), \lambda(t)+\frac{3}{2\alpha} t\dot{\lambda}(t))\rangle.
		\end{aligned}
	\end{equation*} 
	Thus, by substituting the above three relations into \eqref{equ4} and using $Ax^*=b$, we obtain
	\begin{equation}\label{equ11}
		\begin{aligned}
			\dot{E}_{\mu}(t)= &(2t\xi(t)+t^2\dot{\xi}(t))(\mathcal{L}_{\rho}(x(t)+\beta(t)\dot{x}(t),\mu)-\mathcal{L}_{\rho}(x^*,\mu))\\
			&+t^2\xi(t)(\dot{\beta}(t)-\frac{\alpha}{t}\beta(t))\langle \nabla_x\mathcal{L}_{\rho}(x(t)+\beta(t)\dot{x}(t), \lambda(t)+\frac{3}{2\alpha} t\dot{\lambda}(t)),\dot{x}(t)\rangle \\
			&-t^2\xi(t)^2\beta(t)\Vert \nabla_x\mathcal{L}_{\rho}(x(t)+\beta(t)\dot{x}(t), \lambda(t)+\frac{3}{2\alpha} t\dot{\lambda}(t))\Vert ^2\\
			&+ \frac{3-\alpha}{3}t(\Vert \dot{x}(t)\Vert^2+\Vert \dot{\lambda}(t)\Vert^2)
			-\frac{2\alpha}{3}t\xi(t)\langle \nabla_x\mathcal{L}_{\rho}(x(t)+\beta(t)\dot{x}(t), \mu), x(t)-x^*\rangle\\
			&+\frac{2\alpha}{3}t\xi(t)\beta(t)\langle \lambda(t)+\frac{3}{2\alpha} t\dot{\lambda}(t)-\mu, A\dot{x}(t)\rangle .
		\end{aligned}
	\end{equation}
	It follows from the convexity of $f$ that
	\begin{equation}\label{equ-lem2.1-1}
		\begin{aligned}
			&-\langle \nabla_x\mathcal{L}_{\rho}(x(t)+\beta(t)\dot{x}(t),\mu), x(t)-x^*\rangle\\
			\leq & \mathcal{L}_{\rho}(x^*,  \mu)-\mathcal{L}_{\rho}(x(t)+\beta(t)\dot{x}(t), \mu)-\frac{\rho}{2}\Vert A(x(t)+\beta(t)\dot{x}(t))-b\Vert^2\\
			&
			+\beta(t)\langle \nabla_x\mathcal{L}_{\rho}(x(t)+\beta(t)\dot{x}(t),\lambda(t)+\frac{3}{2\alpha} t\dot{\lambda}(t)),\dot{x}(t)\rangle
			+\beta(t)\langle \mu-(\lambda(t)+\frac{3}{2\alpha} t\dot{\lambda}(t)),A\dot{x}(t)\rangle.
		\end{aligned}
	\end{equation}
	Furthermore, by using Young's inequality, we estimate that
	\begin{equation}\label{equ-lem2.1-2}
		\begin{aligned}
			&t\xi(t)(t\dot{\beta}(t)-\frac{\alpha}{3}\beta(t))\langle \nabla_x\mathcal{L}_{\rho}(x(t)+\beta(t)\dot{x}(t), \lambda(t)+\frac{3}{2\alpha} t\dot{\lambda}(t)),\dot{x}(t)\rangle\\
			\leq & \frac{1}{2}\vert\dot{\beta}(t)-\frac{\alpha\beta(t)}{3t}\vert(t^{\frac{5}{2}}\xi(t)^2\Vert \nabla_x\mathcal{L}_{\rho}(x(t)+\beta(t)\dot{x}(t), \lambda(t)+\frac{3}{2\alpha}t\dot{\lambda}(t))\Vert^2+t^{\frac{3}{2}}\Vert \dot{x}(t)\Vert^2).
		\end{aligned}
	\end{equation}
	What's more, taking into account that $\beta(t)=\gamma+\frac{\beta}{t}$ with $\gamma>0, \beta\in\mathbb{R}$ or $\gamma=0, \beta\ge 0$, we obtain the following estimates
	\begin{equation}\label{equ44}
		\begin{cases}
			\text{if } \gamma>0 \text{ and }\beta\in\mathbb{R}, \text{ then exits } t'_2\ge t_1\text{ and } p_1>0 \text{ such that } \\
			\qquad\qquad\qquad-t^2(\beta(t)- \frac{\sqrt{t}}{2}\vert\dot{\beta}(t)-\frac{\alpha\beta(t)}{3t}\vert) \leq -p_1 t^2\text{ for all } t\ge t'_2,\\
			\text{if } \gamma=0 \text{ and }\beta\ge 0, \text{ then exits } t'_2\ge t_1\text{ and } p_1>0 \text{ such that } \\
			\qquad\qquad\qquad-t^2(\beta(t)- \frac{\sqrt{t}}{2}\vert\dot{\beta}(t)-\frac{\alpha\beta(t)}{3t}\vert) \leq -p_1 \beta t\text{ for all } t\ge t'_2.
		\end{cases}	
	\end{equation}
	Furthermore, since $\alpha>3$, there exist $t''_2\ge t_1$ and $p_2>0$ such that 
	\begin{equation}\label{equ101_1}
		\frac{3-\alpha}{3}t+\frac{t^{\frac{3}{2}}}{2}\left\vert\dot{\beta}(t)-\frac{\alpha\beta(t)}{3t}\right\vert\leq -p_2 t \text{ for all } t\ge t''_2.
	\end{equation}
	Therefore, by combining \eqref{equ11}-\eqref{equ101_1} with \eqref{equ15}, we get the the estimate \eqref{equ17} with $t_2=\max\left\{t_1, t'_2, t''_2\right\}$. This completes the proof of Lemma \ref{lem2.1}.
\end{proof}

\begin{lemma}\label{lem2.2}
	Let $t_0>0$, $(x(t), \lambda(t))$ be a solution of the dynamical system $\rm{\eqref{equ2}}$ and $(x^*, \mu)\in\mathbb{F}\times\mathcal{Y}$. Suppose that the assumptions of Lemma \ref{lem2.1} hold. Define $\mu(\cdot):[t_2,+\infty)\to \mathcal{Y}$ by
	\begin{equation}\label{equ102_1}
		\mu(s):=
		\begin{cases}
			\lambda^*+\frac{A(x(s)+\beta(s)\dot{x}(s))-b}{\Vert A(x(s)+\beta(s)\dot{x}(s))-b \Vert},&\text{if } A(x(s)+\beta(s)\dot{x}(s))-b\neq 0,\\
			\lambda^*, &\text{otherwise}.
		\end{cases}
	\end{equation}
	Then, $E_{\mu(s)}$ is bounded on $\left[t_2, +\infty\right)$.
\end{lemma}
\begin{proof}
	By the definition of $\mathcal{L}_{\rho}$, we have
	\begin{align}
		&\mathcal{L}_{\rho}(x(t)+\beta(t)\dot{x}(t), \mu)-\mathcal{L}_{\rho}(x^*, \mu)\nonumber\\
		=& f(x(t)+\beta(t)\dot{x}(t))+\langle \lambda^*, A(x(t)+\beta(t)\dot{x}(t))-b\rangle \nonumber 
		+\frac{\rho}{2}\Vert A(x(t)+\beta(t)\dot{x}(t))-b\Vert^2 \nonumber\\ 
		&+\langle \mu -\lambda^*, A(x(t)+\beta(t)\dot{x}(t))-b\rangle \nonumber\\
		=& \mathcal{L}_{\rho}(x(t)+\beta(t)\dot{x}(t), \lambda^*)-\mathcal{L}_{\rho}(x^*, \lambda^*)\nonumber
		+\langle \mu -\lambda^*, A(x(t)+\beta(t)\dot{x}(t))-b\rangle \nonumber\\
		\ge &\, \langle \mu -\lambda^*, A(x(t)+\beta(t)\dot{x}(t))-b\rangle \label{equ33}.
	\end{align}
	Furthermore, from Lemma \ref{lem2.1}, it follows that $E_{\lambda^*}(t)$ is bounded. Then based on the definition and the boundedness of $E_{\lambda^*}$, we derive that there exists a constant $\widetilde{C}$ such that
	\begin{align}
		&\frac{1}{2}\Vert \frac{2\alpha}{3}((x(t), \lambda(t))- (x^*, \mu))+t(\dot{x}(t), \dot{\lambda}(t))\Vert^2+\frac{\alpha(\alpha-3)}{9}\Vert ((x(t), \lambda(t))- (x^*, \mu))\Vert^2 \nonumber\\
		\leq & \Vert \frac{2\alpha}{3}((x(t), \lambda(t))- (x^*, \lambda^*))+t(\dot{x}(t), \dot{\lambda}(t))\Vert^2\nonumber\\
		&+\frac{2\alpha(\alpha-3)}{9}\Vert ((x(t), \lambda(t))- (x^*, \lambda^*))\Vert^2
		+(\frac{2\alpha}{3}+\frac{2\alpha(\alpha-3)}{9})\Vert \mu -\lambda^*\Vert^2\label{equ35}\\
		\leq & \frac{1}{2}\Vert \frac{2\alpha}{3}((x(t),  \lambda(t))- (x^*, \lambda^*))+t(\dot{x}(t), \dot{\lambda}(t))\Vert^2\nonumber\\
		&+\frac{\alpha(\alpha-3)}{9}\Vert ((x(t), \lambda(t))- (x^*,  \lambda^*))\Vert^2
		+\widetilde{C}+\frac{2\alpha}{3}+\frac{2\alpha(\alpha-3)}{9} \nonumber .
	\end{align}
	Hence, by combining \eqref{equ33}-\eqref{equ35} with the boundedness of $E_{\lambda^*}(t)$, we obtain that there exists a constant $C$ such that for all $\mu\in B(\lambda^*; 1)$:
	\begin{align}
		E_{\mu}(t)
		\leq C+t^2\xi(t)\langle \mu -\lambda^*, A(x(t)+\beta(t)\dot{x}(t))-b\rangle. \label{equ34}
	\end{align}
	According to the definition of $\mu(s)$, it is evident that $\mu(s)\in B(\lambda^*; 1)$. Then, \eqref{equ33} combined with \eqref{equ17}  ensures that for all $t\ge t_2$, we have
	\begin{equation*}
		\begin{aligned}
			\dot{E}_{\mu(s)}(t)\leq & -\delta t\xi(t)(\mathcal{L}_{\rho}(x(t)+\beta(t)\dot{x}(t), \mu(s))-\mathcal{L}_{\rho}(x^*, \mu(s)))\\
			\leq & -\delta t\xi(t)\langle \mu(s) -\lambda^*, A(x(t)+\beta(t)\dot{x}(t))-b\rangle.
		\end{aligned}
	\end{equation*}
	Multiplying both sides of the above inequality by $t^\delta$ and using integration by parts yield the following result:
	\begin{equation*}
		\begin{aligned}
			t^{\delta}E_{\mu(s)}(t)- t_2^{\delta}E_{\mu(s)}(t_2)-&\delta\int_{t_2}^{t}\tau^{\delta-1}E_{\mu(s)}(\tau)d\tau
			\leq -\delta\int_{t_2}^{t}\tau^{\delta+1}\xi(\tau)\langle \mu (s)-\lambda^*, A(x(\tau)+\beta(\tau)\dot{x}(\tau))-b\rangle d\tau,
		\end{aligned}
	\end{equation*}
	which combing with \eqref{equ34} leads to
	\begin{equation}
		\begin{aligned}
			t^{\delta}E_{\mu(s)}(t)\leq &  t_2^{\delta}E_{\mu(s)}(t_2)+\delta\int_{t_2}^{t}\tau^{\delta-1}E_{\mu(s)}(\tau)d\tau
			-\delta\int_{t_2}^{t}\tau^{\delta+1}\xi(\tau)\langle \mu (s)-\lambda^*, A(x(\tau)+\beta(\tau)\dot{x}(\tau))-b\rangle d\tau\\
			\leq & t_2^{\delta}E_{\mu(s)}(t_2)+\delta C\int_{t_2}^{t}\tau^{\delta-1}d\tau
			\leq  t^{\delta}(E_{\mu(s)}(t_2)+C).
		\end{aligned}
	\end{equation}
	Thus,
	\begin{equation}\label{equ36}
		\begin{aligned}
			E_{\mu(s)}(t)\leq E_{\mu(s)}(t_2)+C.
		\end{aligned}
	\end{equation}
	
	Next, we will prove that $ \sup_{\mu\in B(\lambda^*;1)} E_{\mu}(t_2)<+\infty $. First, by the definition of augmented Lagrange function and the Cauchy-Schwarz inequality, we obtain  for any $\mu\in B(\lambda^*;1)$,
	\begin{equation*}
		\begin{aligned}
			&\mathcal{L}_{\rho}(x(t_2)+\beta(t_2)\dot{x}(t_2),\mu)-\mathcal{L}_{\rho}(x^*,  \mu)\\
			= & f(x(t_2)+\beta(t_2)\dot{x}(t_2))-f(x^*)+\langle \mu, A(x(t_2)+\beta(t_2)\dot{x}(t_2))-b\rangle
			+\frac{\rho}{2}\Vert A(x(t_2)+\beta(t_2)\dot{x}(t_2))-b \Vert^2\\
			\leq &f(x(t_2)+\beta(t_2)\dot{x}(t_2))-f(x^*)+\Vert \mu\Vert \Vert A(x(t_2)+\beta(t_2)\dot{x}(t_2))-b\Vert 
			+\frac{\rho}{2}\Vert A(x(t_2)+\beta(t_2)\dot{x}(t_2))-b \Vert^2\\
			\leq & C_1:=(\vert f(x(t_2)+\beta(t_2)\dot{x}(t_2))-f(x^*) \vert +(1+\Vert \lambda^*\Vert)\Vert A(x(t_2)+\beta(t_2)\dot{x}(t_2))-b\Vert 
			+\frac{\rho}{2}\Vert A(x(t_2)+\beta(t_2)\dot{x}(t_2))-b \Vert^2).
		\end{aligned}
	\end{equation*}
	And according to \eqref{equ35}, we also get that for any $\mu\in B(\lambda^*;1)$,
	\begin{equation*}
		\begin{aligned}
			&\frac{1}{2}\Vert \frac{2\alpha}{3}((x(t_2), \lambda(t_2))- (x^*,  \mu))+t_2(\dot{x}(t_2), \dot{\lambda}(t_2))\Vert^2+\frac{\alpha(\alpha-3)}{9}\Vert ((x(t_2),  \lambda(t_2))- (x^*, \mu))\Vert^2\\
			\leq & \Vert \frac{2\alpha}{3}((x(t_2), \lambda(t_2))- (x^*,  \lambda^*))+t_2(\dot{x}(t_2), \dot{\lambda}(t_2))\Vert^2\nonumber\\
			&+\frac{2\alpha(\alpha-3)}{9}\Vert ((x(t_2),  \lambda(t_2))- (x^*,  \lambda^*))\Vert^2
			+(\frac{2\alpha}{3}+\frac{2\alpha(\alpha-3)}{9})\Vert \mu -\lambda^*\Vert^2\\
			\leq & C_2:=(\Vert \frac{2\alpha}{3}((x(t_2), \lambda(t_2))- (x^*, \lambda^*))+t_2(\dot{x}(t_2),  \dot{\lambda}(t_2))\Vert^2\nonumber\\
			&+\frac{2\alpha(\alpha-3)}{9}\Vert ((x(t_2),  \lambda(t_2))- (x^*,  \lambda^*))\Vert^2
			+\frac{2\alpha}{3}+\frac{2\alpha(\alpha-3)}{9}).
		\end{aligned}
	\end{equation*}
	It follows from the above two relations that
	\begin{equation*}
		\sup_{\mu\in B(\lambda^*;1)} E_{\mu}(t_2)\leq t_2^2\xi(t_2)C_1+C_2<+\infty,
	\end{equation*}
	which together with  \eqref{equ36} yields the conclusion. This completes the proof of Lemma \ref{lem2.2}.
\end{proof}

Building on the Lemma \ref{lem2.1} and the Lemma \ref{lem2.2}, we now derive the fast convergence properties of the dynamical system \eqref{equ2} under suitable conditions.

\begin{theorem}\label{thm2.1}
	Suppose that the assumptions of Lemma \ref{lem2.1} hold. Then, for any trajectory $(x(t), \lambda(t))$ of the dynamical system \eqref{equ2} and any primal-dual optimal solution $(x^*, \lambda^*) \in\Omega$ of the problem \eqref{equ1}, $\quad (x(\cdot), \lambda(\cdot))\text{ is bounded}$ and the following conclusions hold:
	
	\vspace{2mm}
	\noindent\textbf{(Integral estimates)}
	\begin{equation*}
		\begin{cases}
			\int_{t_0}^{+\infty} t\xi(t)(\mathcal{L}_{\rho}(x(t)+(\gamma+\frac{\beta}{t})\dot{x}(t), \lambda^*)-\mathcal{L}_{\rho}(x^*, \lambda^*))dt<+\infty, \vspace{2mm}\\
			\int_{t_0}^{+\infty}t(\Vert \dot{x}(t)\Vert^2+\Vert \dot{\lambda}(t)\Vert^2) dt<+\infty, \vspace{2mm}\\
			\int_{t_0}^{+\infty}t\xi(t)\Vert A(x(t)+(\gamma+\frac{\beta}{t})\dot{x}(t))-b\Vert^2 dt<+\infty;
		\end{cases}
	\end{equation*}
	
	\noindent\textbf{(Pointwise estimates)}  
	\begin{equation*}
		\begin{cases}
			\mathcal{L}_{\rho}(x(t)+(\gamma+\frac{\beta}{t})\dot{x}(t), \lambda^*)-\mathcal{L}_{\rho}(x^*, \lambda^*)=\mathcal{O}(\frac{1}{t^2\xi(t)}),\\
			\Vert (\dot{x}(t), \dot{\lambda}(t))\Vert=\mathcal{O}(\frac{1}{t}),\\
			\Vert A(x(t)+(\gamma+\frac{\beta}{t})\dot{x}(t))-b\Vert=\mathcal{O}(\frac{1}{t^2\xi(t)}),\\
			\vert f(x(t)+(\gamma+\frac{\beta}{t})\dot{x}(t))-f(x^*)\vert =\mathcal{O}(\frac{1}{t^2\xi(t)}).
		\end{cases}
	\end{equation*}
\end{theorem}
\begin{proof}
	Integrating \eqref{equ17} with $\mu=\lambda^*$ on $[t_2, t]$, for any $t\ge t_2$, we have
	\begin{equation}\label{equ18}
		\begin{aligned}
			E_{\lambda^*}(t)&+\delta\int_{t_2}^{t} s\xi(s)(\mathcal{L}_{\rho}(x(s)+\beta(s)\dot{x}(s), \lambda^*)-\mathcal{L}_{\rho}(x^*, \lambda^*))ds\\
			&+p_2\int_{t_2}^{t}s\Vert \dot{x}(s)\Vert^2 ds
			+\frac{\rho\alpha}{3}\int_{t_2}^{t}s\xi(s)\Vert A(x(s)+\beta(s)\dot{x}(s))-b\Vert^2 ds+ \frac{\alpha-3}{3}\int_{t_2}^{t}s\Vert \dot{\lambda}(s)\Vert^2 ds\\
			\leq & E_{\lambda^*}(t_2).
		\end{aligned}
	\end{equation}     
	Thus, we deduce that
	\begin{equation}\label{equ43}
		E_{\lambda^*}(t) \text{ is bounded},
	\end{equation}
	and all integral estimates are hold.
	Furthermore, according to the definition of \eqref{equ3}, we get
	\begin{eqnarray}
		&&\mathcal{L}_{\rho}(x(t)+\beta(t)\dot{x}(t), \lambda^*)-\mathcal{L}_{\rho}(x^*, \lambda^*)=\mathcal{O}(\frac{1}{t^2\xi(t)}),\nonumber\\
		&&(x(t), \lambda(t))\text{ is bounded on} \left[t_2, +\infty\right),\label{equ19}\\
		&&\sup_{t\ge t_0}\Vert \frac{2\alpha}{3}((x(t), \lambda(t))-(x^*, \lambda^*))+t(\dot{x}(t),  \dot{\lambda}(t))\Vert^2<+\infty.\label{equ20}
	\end{eqnarray}
	Combining \eqref{equ20} with \eqref{equ19} leads to
	\begin{equation}
		\Vert (\dot{x}(t), \dot{\lambda}(t))\Vert=\mathcal{O}(\frac{1}{t}).
	\end{equation}
	
	In the following,  we further analyze the convergence rates of the feasibility violation and the objective residual by incorporating the conclusion from Lemma \ref{lem2.2}.
	
	\eqref{equ36} combined with the conclusion in Lemma \ref{lem2.2} and the definition of the augmented Lagrange function ensures that for all $t\ge t_2$,
	\begin{equation*}
		\begin{aligned}
			&t^2\xi(t)(f(x(t)+\beta(t)\dot{x}(t))-f(x^*)
			+\langle \mu(s), A(x(t)+\beta(t)\dot{x}(t))-b\rangle)
			\leq C_3:= \sup_{\mu\in B(\lambda^*;1)} E_{\mu}(t_2)+C.
		\end{aligned}
	\end{equation*}
	Then, the above inequality is fulfilled also for $t:=s\ge t_2$. By the definition of $\mu(s)$, if $A(x(s)+\beta(s)\dot{x}(s))-b\neq 0$, we get
	\begin{equation}\label{equ40}
		\begin{aligned}
			&s^2\xi(s)(\mathcal{L}(x(s)+\beta(s)\dot{x}(s), \lambda^*)-\mathcal{L}(x^*,  \lambda^*)
			+\Vert A(x(s)+\beta(s)\dot{x}(s))-b\Vert)\leq C_3.
		\end{aligned}
	\end{equation}
	Since $\mathcal{L}(x(s)+\beta(s)\dot{x}(s), \lambda^*)-\mathcal{L}(x^*, \lambda^*)\ge 0$ and $s\ge t _2$ has been arbitrarily chosen, the above inequality is hold for all $t\ge t_2$. Then, we have for all $t\ge t_2$,
	\begin{equation*}
		\Vert A(x(t)+\beta(t)\dot{x}(t))-b\Vert=\mathcal{O}(\frac{1}{t^2\xi(t)}).
	\end{equation*}
	Since \eqref{equ40} is hold for all $s:=t\ge t_2$, it follows from the above estimate that
	\begin{equation*}
		\begin{aligned}
			f(x(t)+\beta(t)\dot{x}(t))-f(x^*)
			\leq &\frac{C_3}{t^2\xi(t)}-\langle \lambda^*, A(x(t)+\beta(t)\dot{x}(t))-b\rangle\\
			\leq& \frac{C_3}{t^2\xi(t)}+\Vert \lambda^*\Vert \Vert A(x(t)+\beta(t)\dot{x}(t))-b\Vert 
			\leq \frac{(1+\Vert \lambda^*\Vert)C_3}{t^2\xi(t)}.
		\end{aligned}
	\end{equation*}
	On the other hand, the convexity of $f$ and the fact that $(x^*, \lambda^*)\in \Omega$ imply that for all $t\ge t_2$,
	\begin{equation*}
		\begin{aligned}
			f(x(t)+\beta(t)\dot{x}(t))-f(x^*)
			\ge & \langle \nabla f(x^*), x(t)+\beta(t)\dot{x}(t)-x^*\rangle \\
			=&-\langle A^T\lambda^*, x(t)+\beta(t)\dot{x}(t)-x^*\rangle
			=-\langle \lambda^*, A(x(t)+\beta(t)\dot{x}(t))-b\rangle\\
			\ge &-\Vert \lambda^*\Vert \Vert A(x(t)+\beta(t)\dot{x}(t))-b\Vert \ge -\frac{\Vert \lambda^*\Vert C_3}{t^2\xi(t)}.
		\end{aligned}
	\end{equation*}
	Therefor, for all $t\ge t_2$, we have 
	\begin{equation*}
		\vert f(x(t)+\beta(t)\dot{x}(t))-f(x^*)\vert =\mathcal{O}(\frac{1}{t^2\xi(t)}).
	\end{equation*}
	This completes the proof of Theorem \ref{thm2.1}.
\end{proof}

\begin{remark}
	He et al. \cite{bib20} introduced a mixed primal-dual dynamical system incorporating explicit Hessian-driven damping and established its favorable convergence properties. As demonstrated in Theorem \ref{thm2.1}, when $\alpha > 3$, the proposed dynamical system \eqref{equ2} attains a convergence rate of $\mathcal{O}(\frac{1}{t^2\xi(t)})$, matching that of the dynamical system in \cite{bib20}. Moreover, numerical experiments will show that dynamical system \eqref{equ2} achieves comparable suppression of trajectory oscillations while avoiding the explicit Hessian computation required by \cite{bib20}.
\end{remark}
\begin{remark}
	In particular, when $\gamma = \beta = 0$, the dynamical system \eqref{equ2} reduces to a form similar to that proposed by Bo\c{t} et al. in \cite{bib16}. It can also be interpreted as a scaled version of the dynamical system (PD-AVD) with $\theta = \frac{3}{2\alpha}$. When $\xi(t) = 1$, the dynamical system \eqref{equ2} exhibits the same  convergence properties as the dynamical system (PD-AVD). In the special case $A = 0$ and $b = 0$, the dynamical system \eqref{equ2} coincides with the one studied by Alecsa et al. in \cite{bib19}, and Theorem \ref{thm2.1} recovers the convergence rates in \cite[Theorem 12]{bib19}. 
\end{remark}

\section{Primal-dual dynamical system (3) with Tikhonov regularization}\label{sec3}

In Section~\ref{sec2}, we establish the boundedness of the trajectory and the fast convergence properties of the dynamical system \eqref{equ2}. However, the global convergence of the primal trajectory $x(t)$, particularly whether it converges to a solution of the problem \eqref{equ1}, cannot be guaranteed. It is well known that incorporating a Tikhonov regularization term into the dynamical system can often ensure strong convergence of the trajectory (see, e.g., \cite{bib40,bib18-1,bib36}). To enhance the strong convergence property, we introduce a Tikhonov regularization term $\epsilon(t)x(t)$ into the primal-dual dynamical system \eqref{equ2}, resulting in a modified dynamical system that combines implicit Hessian damping with Tikhonov regularization:
\begin{equation}\label{equ2_1}
	\begin{cases}
		\begin{aligned}
			&\ddot{x}(t) + \frac{\alpha}{t} \dot{x}(t) + \xi(t) (\nabla_x\mathcal{L}_{\rho}(\hat{x}(t),\bar{\lambda}(t)) + {\epsilon(t) x(t)}) = 0, \\
			&\ddot{\lambda}(t) + \frac{\alpha}{t} \dot{\lambda}(t) - \xi(t) (\nabla_{\lambda}\mathcal{L}_{\rho}(\bar{x}(t),\lambda(t))+\eta(t)A\dot{x}(t)-\frac{3}{2\alpha} t\xi(t)\beta(t)A(\nabla_x\mathcal{L}_{\rho}(\hat{x}(t), \bar{\lambda}(t))+{\epsilon(t) x(t)}))= 0,
		\end{aligned}
	\end{cases} 
\end{equation}
where $\epsilon : [t_0, +\infty) \rightarrow \mathbb{R}+$ is a $\mathcal{C}^1$ non-increasing function satisfying $\lim\limits_{t\rightarrow +\infty}\epsilon(t) = 0$. All other parameters and trajectories are defined as in Section~\ref{sec2}.  We will demonstrate that the new dynamical system \eqref{equ2_1} not only retains the desirable convergence properties of the dynamical system \eqref{equ2}, such as rapid convergence and reduced oscillations, but also ensures the strong convergence of the trajectory. Detailed convergence analysis is provided in this section, while the existence and uniqueness of a global solution to the associated Cauchy problem is addressed in Appendix~\ref{Appendix A}.

\subsection{Fast convergence rates}
In this subsection, we analyze the fast convergence properties of the dynamical system \eqref{equ2_1}, which indeed exhibits the same asymptotic convergence behavior as the dynamical system \eqref{equ2}. Firstly, we make the following assumption about the Tikhonov regularization parameter $\epsilon(t)$:
\vspace{2mm}

\noindent\textbf{Assumption 3.1}: There exist $M>1$ and $t'_1\ge t_0$ such that
\begin{equation}\label{equ16_3}
	2\dot{\epsilon}(t)+M|(\gamma+\frac{\beta}{t})|\xi(t)\epsilon(t )^2\leq 0,\;(\forall t\ge t'_1)
\end{equation}
where $\gamma>0, \beta\in\mathbb{R}$ or $\gamma=0, \beta\ge 0$.

The assumption is a natural condition that was previously introduced in \cite{bib37}. Moreover, this assumption is easily satisfied. A typical example is given by $\epsilon(t)=\frac{a}{t^r}$ with $a>0$ and $r\ge 1$. Based on this standard assumption, we proceed to analyze the fast convergence rates of the dynamical system \eqref{equ2_1}.

\begin{lemma}\label{lem3.1}
	Let $(x(t), \lambda(t))$ be a solution of the dynamical system $\rm{\eqref{equ2_1}}$ and $(x^*, \mu)\in\mathbb{F}\times\mathcal{Y}$. Denote
	\begin{equation}\label{equ-lem3.1-1}
		E_{\mu}^{\epsilon}(t)=E_{\mu}(t)+\frac{t^2\xi(t)\epsilon(t)}{2}\Vert x(t)\Vert^2,
	\end{equation}
	where $E_{\mu}(t)$ is defined by $\eqref{equ3}$. Assume that $t_0>0$, $\alpha>3$, $\xi$ is a differentiable function that satisfies the condition: there exist $t'_1\ge t_0$ and a sufficiently small constant $\delta>0$ such that \eqref{equ15} holds, and $\epsilon$ is a continuously differentiable and nonincreasing function  that satisfies the \textbf{Assumption 3.1}. Then, there exist $t_2\ge t_0$ and $p_2>0$  such that for all $t\ge t_2$,
	\begin{equation}\label{equ17_3}
		\begin{aligned}
			\dot{E}_{\mu}^{\epsilon}(t)&+\delta t\xi(t)(\mathcal{L}_{\rho}(x(t)+\beta(t)\dot{x}(t), \mu)-\mathcal{L}_{\rho}(x^*, \mu))
			+\delta t\xi(t)\frac{\epsilon(t)}{2}\Vert x(t)\Vert^2\\
			&+\frac{\rho\alpha}{3}t\xi(t)\Vert A(x(t)+\beta(t)\dot{x}(t))-b\Vert^2
			+p_2t\Vert \dot{x}(t)\Vert^2+ \frac{\alpha-3}{3}t\Vert \dot{\lambda}(t)\Vert^2
			\leq \frac{\alpha}{3}t\xi(t)\epsilon(t)\Vert x^*\Vert^2.
		\end{aligned}
	\end{equation}
\end{lemma}
\begin{proof}
	Differentiating \eqref{equ-lem3.1-1} with respect to $t$ yields
	\begin{equation}\label{equ-lem3.1-2}
		\dot{E}_{\mu}^{\epsilon}(t)=\dot{E}_{\mu}(t)+\frac{1}{2}(2t\xi(t)\epsilon(t)+t^2\dot{\xi}(t)\epsilon(t)+t^2\xi(t)\dot{\epsilon}(t))\Vert x(t)\Vert^2+\frac{t^2\xi(t)\epsilon(t)}{2}\langle x(t),\dot{x}(t)\rangle.
	\end{equation}
	Since $f(x)+\frac{\epsilon(t)}{2}\Vert x\Vert^2$ is an $\epsilon(t)$-strongly convex function, we obtain
	\begin{equation}\label{equ5_3}
		\begin{aligned}
			&f(x^*)+\frac{\epsilon(t)}{2}\Vert x^*\Vert^2-f(x(t)+\beta(t)\dot{x}(t))-\frac{\epsilon(t)}{2}\Vert x(t)+\beta(t)\dot{x}(t)\Vert^2\\
			\ge &\langle \nabla f(x(t)+\beta(t)\dot{x}(t))+\epsilon(t)(x(t)+\beta(t)\dot{x}(t)), x^*-(x(t)+\beta(t)\dot{x}(t))\rangle +\frac{\epsilon(t)}{2}\Vert  x(t)+\beta(t)\dot{x}(t)-x^*\Vert^2,
		\end{aligned} 
	\end{equation}
	In addition that, an easy computation shows that
	\begin{equation*}
		\begin{aligned}
			&-\frac{\epsilon(t)}{2}(\Vert x(t)+\beta(t)\dot{x}(t)\Vert^2+\Vert  x(t)+\beta(t)\dot{x}(t)-x^*\Vert^2)\\
			=&-\beta(t)\epsilon(t)\langle \dot{x}(t),  2x(t)+\beta(t)\dot{x}(t)-x^*\rangle 
			-\frac{\epsilon(t)}{2}(\Vert x(t)\Vert^2+\Vert x(t)-x^*\Vert^2),
		\end{aligned}
	\end{equation*}
	which together with \eqref{equ5_3} yields
	\begin{equation}\label{equ12_3}
		\begin{aligned}
			&-\langle \nabla_x\mathcal{L}_{\rho}(x(t)+\beta(t)\dot{x}(t),\mu)+\epsilon(t)x(t), x(t)-x^*\rangle\\
			\leq & \mathcal{L}_{\rho}(x^*,  \mu)-\mathcal{L}_{\rho}(x(t)+\beta(t)\dot{x}(t), \mu)-\frac{\rho}{2}\Vert A(x(t)+\beta(t)\dot{x}(t))-b\Vert^2+\beta(t)\langle \mu-(\lambda(t)+\frac{3}{2\alpha} t\dot{\lambda}(t)),A\dot{x}(t)\rangle\\
			&+\frac{\epsilon(t)}{2}(\Vert x^*\Vert^2-\Vert x(t)\Vert^2-\Vert x(t)-x^*\Vert^2)
			+\beta(t)\langle \nabla_x\mathcal{L}_{\rho}(x(t)+\beta(t)\dot{x}(t),\lambda(t)+\frac{3}{2\alpha} t\dot{\lambda}(t)),\dot{x}(t)\rangle.
		\end{aligned}
	\end{equation}
	Furthermore, by using Young's inequality, we estimate that for all positive constant $M$,
	\begin{equation}\label{equ14_3}
		\begin{aligned}
			&-t^2\xi(t)^2\beta(t)\epsilon(t)\langle \nabla_x\mathcal{L}_{\rho}(x(t)+\beta(t)\dot{x}(t), \lambda(t)+\frac{3}{2\alpha} t\dot{\lambda}(t)), x(t)\rangle\\
			\leq & \frac{1}{2}(\frac{2|\beta(t)|t^2\xi(t)^2}{M}\Vert \nabla_x\mathcal{L}_{\rho}(x(t)+\beta(t)\dot{x}(t), \lambda(t)+\frac{3}{2\alpha} t\dot{\lambda}(t))\Vert^2+\frac{M|\beta(t)|t^2\xi(t)^2\epsilon(t)^2}{2}\Vert x(t)\Vert^2).
		\end{aligned}
	\end{equation}
	Under the assumptions \eqref{equ15} and \eqref{equ16_3}, it holds that for all $t\ge\bar{t}_1:=\max\left\{t_1, t'_1\right\}$
	\begin{equation}\label{equ45_3}
		((2-\frac{2\alpha}{3})t\xi(t)+t^2\dot{\xi}(t))\frac{\epsilon(t)}{2}+t^2\xi(t)\frac{\dot{\epsilon}(t)}{2}+\frac{M|\beta(t)|t^2\xi(t)^2\epsilon(t )^2}{4}\leq -\delta t\xi(t)\frac{\epsilon(t)}{2}.
	\end{equation}
	Finally, substituting \eqref{equ11} into \eqref{equ-lem3.1-2}, and combining the result with \eqref{equ-lem2.1-2}, \eqref{equ44}, \eqref{equ101_1}, \eqref{equ12_3}, \eqref{equ14_3} and \eqref{equ45_3}, we obtain that there exist $p_2>0$ such that for all $t\ge t_2:=\max\left\{\bar{t}_1, t'_2, t''_2\right\}$, the estimate \eqref{equ17_3} holds. This completes the proof of Lemma \ref{lem3.1}.
\end{proof}

\begin{lemma}\label{lem3.2}
	Let $t_0>0$, $(x(t), \lambda(t))$ be a solution of the dynamical system $\rm{\eqref{equ2_1}}$ and $(x^*, \mu)\in\mathbb{F}\times\mathcal{Y}$. Suppose that the assumptions of Lemma \ref{lem3.1} hold. Define $\mu(\cdot):[t_2,+\infty)\to \mathcal{Y}$  as in \eqref{equ102_1}. Then, there exists $\bar{t}_2\ge t_2$ such that $E_{\mu(s)}^{\epsilon}$ is bounded on $\left[\bar{t}_2, +\infty\right)$.
\end{lemma}
\begin{proof}
	Firstly, based on the result of Lemma \ref{lem3.1}, we integrate equation \eqref{equ17_3} with $\mu=\lambda^*$ over the interval $[t_2, t]$. Under the assumption  $\int_{t_0}^{+\infty} t\xi(t)\epsilon(t) dt<+\infty$, it is straightforward to obtain that $E_{\mu}^{\epsilon}(t)$ is bounded. Then, following an argument similar to the proof of Lemma \ref{lem2.2}, we obtain that there exists a constant $C_4$ such that for all $\mu\in B(\lambda^*; 1)$,
	\begin{equation}\label{equ34_3}
		E_{\mu}^{\epsilon}(t)
		\leq C_4+t^2\xi(t)\langle \mu -\lambda^*, A(x(t)+\beta(t)\dot{x}(t))-b\rangle, 
	\end{equation}
	and
	\begin{equation*}
		 \sup_{\mu\in B(\lambda^*;1)} E_{\mu}^{\epsilon}(t_2)<+\infty. 
	\end{equation*}
	Thus, it follows from \eqref{equ17_3} and \eqref{equ33} that 
	\begin{equation*}
		\begin{aligned}
			\dot{E}_{\mu(s)}(t)\leq & -\delta t\xi(t)(\mathcal{L}_{\rho}(x(t)+\beta(t)\dot{x}(t), \mu(s))-\mathcal{L}_{\rho}(x^*, \mu(s)))+\frac{\alpha}{3}t\xi(t)\epsilon(t)\Vert x^*\Vert^2\\
			\leq & -\delta t\xi(t)\langle \mu(s) -\lambda^*, A(x(t)+\beta(t)\dot{x}(t))-b\rangle+\frac{\alpha}{3}t\xi(t)\epsilon(t)\Vert x^*\Vert^2.
		\end{aligned}
	\end{equation*}
	Multiplying both sides of the above inequality by $t^\delta$ and combing with \eqref{equ34_3} yield that
	\begin{equation}\label{equ36_3}
		\begin{aligned}
			E_{\mu(s)}(t)\leq E_{\mu(s)}(t_2)+C_4+\frac{\alpha}{3}\Vert x^*\Vert^2\frac{1}{t^{\delta}}\int_{t_2}^{t}\tau^{\delta+1}\xi(\tau)\epsilon(\tau)d\tau.
		\end{aligned}
	\end{equation}
	What's more, under the assumption $\int_{t_0}^{+\infty} t\xi(t)\epsilon(t)<+\infty$ , Lemma \ref{lemB.2} applied to the functions  $\psi(t) =t^\delta$ and $\phi(t)=t\xi(t)\epsilon(t)$ provides $\lim\limits_{t\rightarrow+\infty}\frac{1}{t^{\delta}}\int_{t_2}^{t}\tau^{\delta+1}\xi(\tau)\epsilon(\tau)d\tau =0$, Hence, there exist $\bar{t}_2\ge t_2$ and $C_5$ such that for all $t\ge\bar{t}_2$,
	\begin{equation*}
		\frac{1}{t^{\delta}}\int_{t_2}^{t}\tau^{\delta+1}\xi(\tau)\epsilon(\tau)d\tau \leq C_5,
	\end{equation*}
	which together with $ \sup_{\mu\in B(\lambda^*;1)} E_{\mu}(t_2)<+\infty $ applied to \eqref{equ36_3} yields the boundedness of $E_{\mu(s)}$ on $\left[\bar{t}_2, +\infty\right)$.  This completes the proof of Lemma \ref{lem3.2}.
\end{proof}

\begin{theorem}\label{thm3.1_3}
	Suppose that the assumptions of Lemma  \ref{lem3.1} hold. Then, for any trajectory $(x(t), \lambda(t))$ of the dynamical system \eqref{equ2_1} and any primal-dual optimal solution $(x^*, \lambda^*) \in\Omega$ of the problem \eqref{equ1}, the following conclusions hold:
	
	\noindent(i) If $\epsilon$ further satisfies $\int_{t_0}^{+\infty}t\xi(t)\epsilon(t) dt<+\infty$ and $\lim\limits_{t\rightarrow+\infty}t^2\xi(t)=+\infty$, then $(x(\cdot), \lambda(\cdot))$ is bounded and the following conclusions hold:\vspace{2mm}
	
	\textbf{(Integral estimates)}
	\begin{equation*}
		\begin{cases}
			\int_{t_2}^{+\infty} t\xi(t)(\mathcal{L}_{\rho}(x(t)+(\gamma+\frac{\beta}{t})\dot{x}(t), \lambda^*)-\mathcal{L}_{\rho}(x^*, \lambda^*))dt<+\infty, \vspace{2mm}\\
			\int_{t_2}^{+\infty}t\xi(t)\epsilon(t)\Vert x(t)\Vert^2 dt<+\infty, \vspace{2mm}\\
			\int_{t_2}^{+\infty}t(\Vert \dot{x}(t)\Vert^2+\Vert \dot{\lambda}(t)\Vert^2) dt<+\infty, \vspace{2mm}\\
			\int_{t_2}^{+\infty}t\xi(t)\Vert A(x(t)+(\gamma+\frac{\beta}{t})\dot{x}(t))-b\Vert^2 dt<+\infty;
		\end{cases}
	\end{equation*}
	
	\textbf{(Pointwise estimates)}  
	\begin{equation*}
		\begin{cases}
			\mathcal{L}_{\rho}(x(t)+(\gamma+\frac{\beta}{t})\dot{x}(t), \lambda^*)-\mathcal{L}_{\rho}(x^*, \lambda^*)=\mathcal{O}(\frac{1}{t^2\xi(t)}),\\
			\Vert (\dot{x}(t), \dot{\lambda}(t))\Vert=\mathcal{O}(\frac{1}{t}),\\
			\Vert A(x(t)+(\gamma+\frac{\beta}{t})\dot{x}(t))-b\Vert=\mathcal{O}(\frac{1}{t^2\xi(t)}),\\
			\vert f(x(t)+(\gamma+\frac{\beta}{t})\dot{x}(t))-f(x^*)\vert =\mathcal{O}(\frac{1}{t^2\xi(t)}).
		\end{cases}
	\end{equation*}
	
	\noindent (ii) If $\epsilon$ further satisfies $\int_{t_0}^{+\infty}\frac{\xi(t)\epsilon(t)}{t} dt<+\infty$, then $\lim\limits_{t\rightarrow+\infty}(\dot{x}(t), \dot{\lambda}(t))=\mathbf{0}$.
\end{theorem}
\begin{proof}
	By proceeding along the lines of the proof of Theorem \ref{thm2.1}, we first integrate \eqref{equ17_3} with $\mu=\lambda^*$. Then, in light of the assumption  $\int_{t_0}^{+\infty}t\xi(t)\epsilon(t) dt<+\infty$ and the definition of $E_{\lambda^*}^{\epsilon}$, we obtain all conclusions in (i), excluding the convergence rates for the constraint violation and the objective residual, whose derivation will be provided at the end.
	
	For the conclusion in (ii), based on the integrated result, and taking into account that $\int_{t_0}^{+\infty}\frac{\xi(t)\epsilon(t)}{t} dt<+\infty$, we deduce that for all $t\ge t_2$,
	\begin{equation}\label{equ21_3}
		\Vert  \frac{2\alpha}{3t}((x(t), \lambda(t))-(x^*, \lambda^*))+(\dot{x}(t), \dot{\lambda}(t))\Vert^2\leq \frac{E_{\lambda^*}^{\epsilon}(t_2)}{t^2}+\frac{\alpha}{3}\frac{\Vert x^*\Vert^2}{t^2}\int_{t_2}^{t}s\xi(s)\epsilon(s) ds,
	\end{equation}
	and 
	\begin{equation}\label{equ22_3}
		\Vert \frac{1}{t}((x(t), \lambda(t))-(x^*, \lambda^*))\Vert^2\leq \frac{E_{\lambda^*}^{\epsilon}(t_2)}{t^2}+\frac{\alpha}{3}\frac{\Vert x^*\Vert^2}{t^2}\int_{t_2}^{t}s\xi(s)\epsilon(s) ds.
	\end{equation}
	Under the assumption $\int_{t_0}^{+\infty} \frac{\xi(t)\epsilon(t)}{t}<+\infty$ , applying Lemma \ref{lemB.2} with $\psi(t) =t^2$, $\phi(t)=\frac{\xi(t)\epsilon(t)}{t}$ to \eqref{equ21_3} and \eqref{equ22_3}, we obtain
	\begin{equation}\label{equ30_3}
		\lim\limits_{t\rightarrow +\infty}\frac{1}{t^2}\int_{t_2}^{t} s^2\frac{\xi(s)\epsilon(s)}{s}ds=0,
	\end{equation}
	which together with \eqref{equ21_3} and \eqref{equ22_3} yields 
	\begin{equation*}
		\lim\limits_{t\rightarrow+\infty}\Vert (\dot{x}(t), \dot{\lambda}(t))\Vert =0,\quad i.e.,\quad \lim\limits_{t\rightarrow+\infty}(\dot{x}(t), \dot{\lambda}(t))=\mathbf{0}.
	\end{equation*}
	For the convergence rates of the constraint violation and the objective residual, we obtain the corresponding convergence rates for these quantities along the trajectory of the dynamical system \eqref{equ2_1} by incorporating the results of Lemma \ref{lem3.2} and following the same analysis used in Theorem \ref{thm2.1} for deriving the convergence rates of the constraint violation and the objective residual. This completes the proof of Theorem \ref{thm3.1_3}.
\end{proof}

\subsection{Strong convergence of the trajectory}

In this subsection, we establish the strong convergence of the trajectory generated by the dynamical system \eqref{equ2_1} to the minimum-norm solution of the problem \eqref{equ1}, under the assumption that the Tikhonov regularization parameter $\epsilon(t)$ vanishes at a suitable rate. To facilitate the analysis, we first recall some key properties of Tikhonov regularization that will play a crucial role in our convergence proof.

Throughout this subsection, and without causing any ambiguity, we denote by $x^*$ the element of minimal norm in the solution set $S$ of the problem \eqref{equ1}, and by $\lambda^*$ the associated optimal dual variable such that $(x^*,\lambda^*)\in \Omega$. Then, $x^*= \rm{Proj}_S 0$.
For any $\epsilon>0$, we denote by $x_\epsilon$ the unique solution of the strongly convex minimization problem
\begin{equation*}
	x_\epsilon:={\rm{arg}}\min_{x\in \mathcal{X}}\mathcal{L}_{\epsilon}(x), \quad\text{where } \mathcal{L}_{\epsilon}(x):=\mathcal{L}_{\rho}(x, \lambda^*)+\frac{\epsilon}{2}\Vert x\Vert^2.
\end{equation*}
The first-order optimality condition leads to
\begin{equation}\label{equ41}
	\nabla\mathcal{L}_{\epsilon}(x_\epsilon)=\nabla_x\mathcal{L}_{\rho}(x_\epsilon,\lambda^*)+\epsilon x_\epsilon=0.
\end{equation}
From the classical properties of the Tikhonov regularization, we have
$\lim\limits_{\epsilon\rightarrow 0}\Vert x_\epsilon-x^* \Vert =0$.
which was originally established by Tikhonov \cite{bib31}. Furthermore, Since $\nabla\mathcal{L}_{\rho}(x, \lambda^*)$ is a monotone operator and $\nabla_x\mathcal{L}_{\rho}(x^*, \lambda^*)=0$ ,   we deduce from the optimality condition \eqref{equ41} that
\begin{equation*}
	\Vert x_\epsilon\Vert\leq\Vert x^*\Vert,\; \text{for all }\epsilon>0.
\end{equation*}

With the classical results in place, we are now in a position to state the strong convergence of the trajectory generated by the the dynamical system \eqref{equ2_1}.
\begin{theorem}\label{thm3.3}
	Suppose that  $\int_{t_0}^{+\infty}\frac{\xi(t)\epsilon(t)}{t} dt<+\infty$  and the assumptions of Lemma \ref{lem3.1} are holds. Assume further that $\xi$ and $\epsilon$ satisfy
	\begin{equation*}
		\lim\limits_{t\rightarrow+\infty}t^2\xi(t)\epsilon(t)=+\infty,\quad \liminf_{t\rightarrow+\infty}\xi(t)\neq 0,\text{ and }\lim\limits_{t\rightarrow+\infty} \frac{1}{t^2\xi(t)\epsilon(t)} \int_{t_0}^{t} s^2\xi(s)^2\epsilon(s)^2 ds=0.
	\end{equation*}
	Let $(x(\cdot), \lambda(\cdot))$ be a global solution of the dynamical system $\rm{\eqref{equ2_1}}$, then 
	\begin{equation*}
		\liminf_{t\rightarrow+\infty}\Vert x(t)-x^*\Vert=0.
	\end{equation*}
	where $x^*=\rm{Proj}_S 0$. What's more, if there exists a large enough T such that the trajectory $\left\{ x(t):t\ge T\right\}$ stays in either the open ball $B(0,\Vert x^*\Vert)$ or its complement, then 
	\begin{equation}\label{equ24}
		\lim\limits_{t\rightarrow+\infty}\Vert x(t)-x^*\Vert=0.
	\end{equation}
\end{theorem}
\begin{proof}
	Depending on three positions of the trajectory $x(\cdot)$ relative to the Ball $B(0,\Vert x^*\Vert)$, we analyze separately the following three situations.
	
	\textbf{Case I:} There exists a large enough T such that $x(t)$ stays in the complement of $B(0,\Vert x^*\Vert)$, i.e.,
	\begin{equation}\label{equ27}
		\Vert x(t)\Vert\ge\Vert x^*\Vert,\;\text{for all }t\ge T.
	\end{equation}
	Define a new energy function $\widetilde{E}:\left[t_0,+\infty\right)\rightarrow\left[0,+\infty\right)$ by
	\begin{equation*}\label{equ27}
		\begin{aligned}
			\widetilde{E}(t)=\frac{E_{\lambda^*}^{\epsilon}(t)}{t^2}-\frac{\xi(t)\epsilon(t)}{2}\Vert x^*\Vert^2.
		\end{aligned}
	\end{equation*}
	A similar argument to the proof of Lemma \ref{lem3.1} shows that
	\begin{equation*}
		\begin{aligned}
			\frac{2}{t}\widetilde{E}+\dot{\widetilde{E}}(t)\leq &(\frac{2(3-\alpha)}{3t}\xi(t)+\dot{\xi}(t))(\mathcal{L}_{\rho}(x(t)+\beta(t)\dot{x}(t), \lambda^*)-\mathcal{L}_{\rho}(x^*, \lambda^*))\\
			&+(\dot{\xi}(t)\frac{\epsilon(t)}{2}+\xi(t)\frac{\dot{\epsilon}(t)}{2}+\frac{3-\alpha}{3t}\xi(t)\epsilon(t))(\Vert x(t)\Vert^2-\Vert x^*\Vert^2)\\
			&+\xi(t)(\dot{\beta}(t)-\frac{\alpha}{3t}\beta(t))\langle \nabla_x\mathcal{L}_{\rho}(x(t)+\beta(t)\dot{x}(t), \lambda(t)+\frac{3}{2\alpha} t\dot{\lambda}(t)),\dot{x}(t)\rangle\\
			&-\xi(t)^2\beta(t)\Vert \nabla_x\mathcal{L}_{\rho}(x(t)+\beta(t)\dot{x}(t), \lambda(t)+\frac{3}{2\alpha} t\dot{\lambda}(t))\Vert ^2\\
			&-\xi(t)^2\beta(t)\epsilon(t)\langle \nabla_x\mathcal{L}_{\rho}(x(t)+\beta(t)\dot{x}(t), \lambda(t)+\frac{3}{2\alpha} t\dot{\lambda}(t)), x(t)\rangle\\
			&+\frac{3-\alpha}{3t}(\Vert \dot{x}(t)\Vert^2+\Vert \dot{\lambda}(t)\Vert^2)-\frac{\alpha}{3t}\xi(t)\epsilon(t)\Vert x(t)-x^*\Vert^2\\
			&-\frac{\rho\alpha}{3t}\xi(t)\Vert A(x(t)+\beta(t)\dot{x}(t))-b\Vert^2.
		\end{aligned}
	\end{equation*}
	Moreover,  by using Young's inequality, we derive that for all $M, q >0$,
	\begin{equation*}
		\begin{aligned}
			&\xi(t)(\dot{\beta}(t)-\frac{\alpha}{3t}\beta(t))\langle \nabla_x\mathcal{L}_{\rho}(x(t)+\beta(t)\dot{x}(t), \lambda(t)+\frac{3}{2\alpha} t\dot{\lambda}(t)),\dot{x}(t)\rangle\\
			\leq & \vert\dot{\beta}(t)-\frac{\alpha\beta(t)}{3t}\vert(\frac{t}{q}\xi(t)^2\Vert \nabla_x\mathcal{L}_{\rho}(x(t)+\beta(t)\dot{x}(t), \lambda(t)+\frac{3}{2\alpha} t\dot{\lambda}(t))\Vert^2+\frac{q}{4t}\Vert \dot{x}(t)\Vert^2),
		\end{aligned}
	\end{equation*}
	and
	\begin{equation*}
		\begin{aligned}
			&-\xi(t)^2\beta(t)\epsilon(t)\langle \nabla_x\mathcal{L}_{\rho}(x(t)+\beta(t)\dot{x}(t), \lambda(t)+\frac{3}{2\alpha} t\dot{\lambda}(t)), x(t)\rangle\\
			\leq & \frac{|\beta(t)|\xi(t)^2}{M}\Vert \nabla_x\mathcal{L}_{\rho}(x(t)+\beta(t)\dot{x}(t), \lambda(t)+\frac{3}{2\alpha} t\dot{\lambda}(t))\Vert^2+\frac{M|\beta(t)|\xi(t)^2\epsilon(t)^2}{4}\Vert x(t)\Vert^2.
		\end{aligned}
	\end{equation*}
	Since according to $\alpha>3$, we know that for the constant $M>1$, there exist constant $q>0$ and $t'_3\ge t_0$ such that for all $t\ge t'_3$,
	\begin{equation*}
		\frac{\vert \beta(t)\vert}{M}-\beta(t)+\frac{t}{q}\vert\dot{\beta}(t)-\frac{\alpha\beta(t)}{3t}\vert \leq 0,
	\end{equation*}
	and there exists $t''_3\ge t_0$ such that for all $t\ge t''_3$,
	\begin{equation*}
		\frac{3-\alpha}{3t}+\frac{q}{4t}\vert\dot{\beta}(t)-\frac{\alpha\beta(t)}{3t}\vert\leq 0.
	\end{equation*}
	Moreover, combining with the assumptions \eqref{equ15} and \eqref{equ16_3}, there exists $t'''_3\ge t_0$ such that for all $t\ge t'''_3$,
	\begin{equation*}
		\xi(t)\frac{\dot{\epsilon}(t)}{2}+\dot{\xi}(t)\frac{\epsilon(t)}{2}+\frac{3-\alpha}{3t}\xi(t)\epsilon(t)+\frac{M\vert\beta(t)\vert\xi(t)^2\epsilon(t)^2}{4}\leq 0.
	\end{equation*}
	Hence by incorporating the above estimates and under assumption \eqref{equ15}, the following result holds: 
	\begin{equation*}
		\frac{2}{t}\widetilde{E}+\dot{\widetilde{E}}(t)\leq\frac{M\vert\beta(t)\vert\xi(t)^2\epsilon(t)^2}{4}\Vert x^*\Vert^2,
	\end{equation*}
	for all $t\ge t_3:=\max\left\{t'_3, t''_3, t'''_3\right\}$.
	By multiplying both sides of the above inequality by $t^2$ and integrating it on $[t_3,t]$ for an arbitrary $t\ge t_3$, and taking into account that $\beta(t)=\gamma+\frac{\beta}{t}$ is bounded, we obtain that there exists a constant $C_6>0$ such that
	\begin{equation*}
		\widetilde{E}(t)\leq \frac{t_3\widetilde{E}(t_3)}{t^2}+\frac{C_6\Vert x^*\Vert^2}{4t^2}\int_{t_3}^{t}s^2\xi(s)^2\epsilon(s)^2ds.
	\end{equation*}
	Moreover, by the $\epsilon(t)$- strong convexity of $\mathcal{L}_{\epsilon(t)}$  and together with \eqref{equ41}, we deduce that 
	\begin{equation*}
		\mathcal{L}_{\epsilon(t)}(x(t))-\mathcal{L}_{\epsilon(t)}(x_{\epsilon(t)})\ge\frac{\epsilon(t)}{2}\Vert x(t)-x_{\epsilon(t)}\Vert^2.
	\end{equation*}
	Using the definition of $\mathcal{L}_{\epsilon}(x)$ and the fact that $ \mathcal{L}_{\rho}(x^*,\lambda^*)\leq\mathcal{L}_{\rho}(x_{\epsilon(t)}, \lambda^*)$, we derive that
	\begin{equation*}
		\begin{aligned}
			\mathcal{L}_{\epsilon(t)}(x^*)-\mathcal{L}_{\epsilon(t)}(x_{\epsilon(t)})
			=&\mathcal{L}_{\rho}(x^*,\lambda^*)-\mathcal{L}_{\rho}(x_{\epsilon(t)},\lambda^*)+\frac{\epsilon(t)}{2}\Vert x^*\Vert^2-\frac{\epsilon(t)}{2}\Vert x_{\epsilon(t)}\Vert^2
			\leq\frac{\epsilon(t)}{2}\Vert x^*\Vert^2-\frac{\epsilon(t)}{2}\Vert x_{\epsilon(t)}\Vert^2.
		\end{aligned}
	\end{equation*}
	As a consequence,
	\begin{equation}\label{equ28}
		\frac{\epsilon(t)}{2}(\Vert x(t)-x_{\epsilon(t)}\Vert^2+\Vert x_{\epsilon(t)}\Vert^2-\Vert x^*\Vert^2)\leq\mathcal{L}_{\epsilon(t)}(x(t))-\mathcal{L}_{\epsilon(t)}(x^*).
	\end{equation}%
	According to the definition of $\widetilde{E}$ and $\Vert x(t)\Vert \ge \Vert x^*\Vert$, we obtain that
	\begin{equation*}
		\begin{aligned}
			\widetilde{E}(t)\ge & \frac{1}{2}\xi(t)\epsilon(t)(\Vert x(t)+\beta(t)\dot{x}(t)-x_{\epsilon(t)}\Vert^2+\Vert x_{\epsilon(t)}\Vert^2-\Vert x^*\Vert^2+\Vert x(t)\Vert^2-\Vert x(t)+\beta(t)\dot{x}(t)\Vert^2)\\
			=& \frac{1}{2}\xi(t)\epsilon(t)(\Vert x(t)-x_{\epsilon(t)}\Vert^2+\Vert x_{\epsilon(t)}\Vert^2-\Vert x^*\Vert^2)
			-\xi(t)\beta(t)\epsilon(t)\langle \dot{x}(t), x_{\epsilon(t)}\rangle .
		\end{aligned}
	\end{equation*}
	Then,
	\begin{equation*}
		\begin{aligned}
			\Vert x(t)-x_{\epsilon(t)}\Vert^2\leq & \Vert x^*\Vert^2-\Vert x_{\epsilon(t)}\Vert^2+2\beta(t)\langle \dot{x}(t), x_{\epsilon(t)}\rangle 
			+\frac{2t_3\widetilde{E}(t_3)}{t^2\xi(t)\epsilon(t)}+\frac{C_6\Vert x^*\Vert^2}{2t^2\xi(t)\epsilon(t)}\int_{t_3}^{t}s^2\xi(s)^2\epsilon(s)^2ds.
		\end{aligned}
	\end{equation*}
	From Theorem \ref{thm3.1_3}, we have $\dot{x}(t)\rightarrow 0$ as $t\rightarrow +\infty$. Since $x_{\epsilon(t)}\rightarrow x^*$, we get that $2\beta(t)\langle \dot{x}(t), x_{\epsilon(t)}\rangle \rightarrow 0$ as $t\rightarrow +\infty$. And since 
	\begin{equation*}
		t^2\xi(t)\epsilon(t)\rightarrow +\infty \text{ and }  \frac{1}{t^2\xi(t)\epsilon(t)} \int_{t_3}^{t} s^2\xi(s)^2\epsilon(s)^2 ds \rightarrow 0 \text{ as } t\rightarrow +\infty,
	\end{equation*}
	and 
	$\lim\limits_{t\rightarrow +\infty}	\Vert x(t)-x_{\epsilon(t)}\Vert=0$,
	we have
	\begin{equation*}
		\lim\limits_{t\rightarrow +\infty}	\Vert x(t)-x^*\Vert=0. 
	\end{equation*}
	
	\textbf{Case II:} Suppose that there exists a large enough T such that $x(t) \in B(0,\Vert x^*\Vert)\;\text{for all }t\ge T$, i.e.,
	\begin{equation}\label{equ29}
		\Vert x(t)\Vert<\Vert x^*\Vert,\;\text{for all }t\ge T.
	\end{equation}
	This implies the existence of a weak cluster point for the trajectory $x(t)$.
	
	Let $\tilde{x}$ be a weak cluster point of trajectory $x(t)$. According to the definition of weak convergence,
	there exists a sequence$(t_n)_{n\in\mathbb{N}}$ with $t_n\rightarrow +\infty$ such that $x(t_n) \rightharpoonup \tilde{x}$ as $n \rightarrow +\infty$. Given that $f$ is continuous, the norm is weakly lower semi-continuous and $A$ is a continuous linear operator, it follows that $L_{\rho}(\cdot, \lambda^*)$ is weakly lower semi-continuous. Consequently,
	\begin{equation}\label{equ31}
		\mathcal{L}_{\rho}(\tilde{x},  \lambda^*)\leq \liminf_{n\rightarrow+\infty}	\mathcal{L}_{\rho}(x(t_n),  \lambda^*).
	\end{equation}
	
	Next, we need to define a new energy function $\widehat{E}:\left[t_0,+\infty\right)\rightarrow\left[0,+\infty\right)$ by
	\begin{equation*}
		\begin{aligned}
			\widehat{E}(t)=\frac{E_{\lambda^*}^{\epsilon}(t)}{t^2}.
		\end{aligned}
	\end{equation*}
	A similar argument to the proof of \textbf{Case I} shows that
	\begin{equation*}
		\frac{2}{t}\widehat{E}+\dot{\widehat{E}}(t)\leq \frac{\alpha}{3t}\xi(t)\epsilon(t)\Vert x^*\Vert^2.
	\end{equation*}
	Multiplying both sides of the above inequality by $t^2$ and integrating it form $t_3$ to $t$ result in
	\begin{equation*}
		\widehat{E}(t)\leq \frac{\widehat{E}(t_3)}{t^2} + \frac{\alpha\Vert x^*\Vert^2}{3t^2}\int_{t_3}^{t}s\xi(s)\epsilon(s)ds.
	\end{equation*}
	Under the assumption $\int_{t_0}^{+\infty} \frac{\xi(t)\epsilon(t)}{t}<+\infty$ , applying Lemma \ref{lemB.2}  to above inequality, we obtain
	\begin{equation}
		\lim\limits_{t\rightarrow +\infty}\widehat{E}(t)=0.
	\end{equation}
	Then, combining the definition of $\widehat{E}$ with $\liminf_{t\rightarrow+\infty}\xi(t)\neq 0$ yields
	\begin{equation*}
		\lim\limits_{t\rightarrow+\infty}\mathcal{L}_{\rho}(x(t)+\beta(t)\dot{x}(t),\lambda^*)-\mathcal{L}_{\rho}(x^*,  \lambda^*)=0,
	\end{equation*}
	which together with the conclusion $\lim\limits_{t\rightarrow+\infty}(\dot{x}(t), \dot{\lambda}(t))=\mathbf{0}$ given in Theorem \ref{thm3.1_3} and  \eqref{equ31} leads to
	\begin{equation*}
		\mathcal{L}_{\rho}(\tilde{x}, \lambda^*)\leq\mathcal{L}_{\rho}(x^*,\lambda^*).
	\end{equation*}
	Considering $(x^*,\lambda^*)\in\Omega$, one has
	$\mathcal{L}_{\rho}(x^*, \lambda^*)\leq\mathcal{L}_{\rho}(\tilde{x},  \lambda^*)\leq\mathcal{L}_{\rho}(x^*,  \lambda^*)$.
	Therefore,
	\begin{equation*}
		\mathcal{L}_{\rho}(\tilde{x}, \lambda^*)=\mathcal{L}_{\rho}(x^*,  \lambda^*)=\min_{x\in\mathcal{X}}\mathcal{L}_{\rho}(x,  \lambda^*),
	\end{equation*}
	which implies
	\begin{equation*}
		\tilde{x}\in{\arg}\min_{x\in\mathcal{X}}\mathcal{L}_{\rho}(x, \lambda^*)\subseteq S.
	\end{equation*}
	By using \eqref{equ29}, we get
	\begin{equation*}
		\limsup_{n\rightarrow+\infty}\Vert x(t_n)\Vert\leq\Vert x^*\Vert.
	\end{equation*}
	On the other hand, since the norm is weakly lower semi-continuous and the trajectory $x(\cdot)$ converges weakly to $\tilde{x}$, then
	\begin{equation*}
		\Vert\tilde{x}\Vert\leq\liminf_{n\rightarrow+\infty}\Vert  x(t_n)\Vert\leq\limsup_{n\rightarrow+\infty}\Vert  x(t_n)\Vert\leq\Vert x^*\Vert.
	\end{equation*}
	Combining $\tilde{x}\in S$ and the definition of $x^*$,  it is concluded that $\tilde{x}=x^*$. Therefore,
	\begin{equation*}
		x(t)\rightharpoonup x^*,\quad {\rm{as}}\quad t\rightarrow+\infty,
	\end{equation*}
	and
	\begin{equation*}
		\Vert x^*\Vert\leq\liminf_{t\rightarrow+\infty}\Vert x(t)\Vert\leq\limsup_{t\rightarrow+\infty}\Vert x(t)\Vert\leq\Vert  x^*\Vert.
	\end{equation*}
	Thus, the strong convergence of the trajectory in this case is established, that is,
	\begin{equation*}
		\lim\limits_{t\rightarrow+\infty}\Vert x(t)-x^*\Vert =0.
	\end{equation*}
	
	\textbf{Case III:} For any $T\ge t_0$, there exist separately $t\ge T$ and $s \ge T$ such that $x(t)\in B(0,\Vert x^*\Vert)$ or $x(s)$ in the complement of the ball $B(0,\Vert x^*\Vert)$. As a consequence, there exists a sequence $(t_n)_{n\in \mathbb{N}}$ such that $t_n\rightarrow+\infty$ is as $n\rightarrow+\infty$ and $\Vert  x(t_n)\Vert =\Vert x^*\Vert \text{ for all } n\in \mathbb{N}$.
	Obviously,
	\begin{equation}\label{equ32}
		\lim\limits_{n\rightarrow+\infty}\Vert x(t_n)\Vert =\Vert x^*\Vert.
	\end{equation}
	Similarly to the proof in Case II, we obtain that $x^*$ is a unique weak cluster point such that $x(t_n)\rightharpoonup x^*,\, {\rm{as}}\, n\rightarrow+\infty$.
	Combining \eqref{equ32} results in
	\begin{equation*}
		\liminf_{t\rightarrow+\infty}\Vert x(t)-x^*\Vert =0.
	\end{equation*}
	This completes the proof of Theorem \ref{thm3.3}.
\end{proof}

\begin{remark}
	In \cite{bib20}, He et al. only established the convergence rates for a mixed primal-dual dynamical system with \emph{explicit} Hessian-driven damping, and were not able to guarantee the strong convergence of the trajectory. However, Theorem \ref{thm3.1_3} and Theorem \ref{thm3.3} show that our dynamical system \eqref{equ2_1} not only achieves the same fast convergence rates as the dynamical system proposed by He et al. \cite{bib20}, but also ensures strong convergence of the trajectory.
\end{remark}
\begin{remark}
	When $\gamma=0$ and $\beta=0$, the dynamical system \eqref{equ2_1} reduces to the one proposed by Zhu et al. in \cite{bib23} with $\theta=\frac{3}{2\alpha}$. Therefore, our dynamical system \eqref{equ2_1} can be viewed as an extension of the dynamical system introduced in \cite{bib23}, incorporating an implicit Hessian damping term. Moreover, in the case where $A=0$ and $b=0$, the dynamical system \eqref{equ2_1} reduces to the dynamical system proposed by Alecsa et al. in \cite{bib37}. The results in Theorem \ref{thm3.1_3} and Theorem \ref{thm3.3} then imply some of the convergence results established in \cite[Theorem 2.1 and Theorem 3.2]{bib37}. However, the convergence rate $o(\frac{1}{t^2})$ of the objective residual obtained in \cite{bib37} cannot be directly derived from Theorem \ref{thm3.1_3}. This is also a direction for future research.
\end{remark}

%
\section{Numerical experiments}\label{sec5}
In this section, to validate the theoretical results of the primal-dual dynamical system with implicit Hessian damping \eqref{equ2} (PDDS-IHD for short) and the primal-dual dynamical system with implicit Hessian damping and Tikhonov regularization \eqref{equ2_1} (PDDS-IHDTR for short), we provide numerical experiments. All experiments were carried out on a MacBook Pro (Apple M2, 16GB memory) using MATLAB R2021b.
%
we consider the convex optimization problem:
\begin{equation}\label{equ52}
	\begin{aligned}
		&\min_{x, y} \; (mx_1+nx_2+ex_3)^2,\\
		&\;{\rm{s.t.}} \; Ax =b,
	\end{aligned}
\end{equation}
where $A= (m, -n, e)$ with $m, n, e\in\mathbb{R}\backslash \left\{0\right\}$ and $b=0$. A direct computation shows that the set of solutions for this problem is given by $S=\left\{(x_1, 0, -\frac{m}{e}x_1)^T\vert\; x_1\in\mathbb{R}\right\}$. Clearly, the minimum-norm solution is $x^*=(0,0,0)^T$, and the corresponding optimal value of the objective function is zero.

First, we conduct two numerical experiments to illustrate the roles of the Tikhonov regularization term $\epsilon(t)x(t)$ and the term $(\gamma+\frac{\beta}{t})\dot{x}(t)$, respectively. In the first experiment, a comparison is made between the dynamical system (PDDS-IHDTR)\eqref{equ2_1}, the dynamical system (PDDS-IHD)\eqref{equ2}, and the dynamical system (PD-AVD) proposed by Bo{\c{t}} et al. \cite{bib16} that does not contain the Tikhonov regularization term $\epsilon(t)x(t)$ and the term $(\gamma+\frac{\beta}{t})\dot{x}(t)$. All dynamical systems are numerically solved with the MATLAB built-in solver ode23 in the interval $[0,50]$.

For  the dynamical system (PDDS-IHDTR) \eqref{equ2_1}, the parameters are set as follows: $m = 5$, $n = 10$, $e = 6$, $\alpha = 3.1$, $\gamma=1$, $\beta =-0.5$, $\rho=1$, $\xi(t) = 1$, $\epsilon(t) = \frac{1}{t^r}$, $r=1.5$, with initial conditions: $x(1)= (1, -1, -1)^T$,  $\lambda(1) = 1$, $\dot{x}(1) = (-1, 1, 1)^T$, $\dot{\lambda}(1) = -1$. For  the dynamical system (PDDS-IHD) \eqref{equ2}, the parameters remain mostly unchanged except that the Tikhonov coefficient is omitted and the same initial conditions are used. For the dynamical system (PD-AVD), we set: $m = 5$, $n = 10$, $e = 6$, $\alpha = 3.1$, $\beta=1$, $\delta=\frac{1}{2.05}$, with identical initial conditions as above.

As shown in \cref{fig: 1_convergence}(a), the trajectory $x(t)$ generated by the dynamical system converges to the minimum-norm solution $x^*$, confirming the theoretical results. In contrast, when the Tikhonov regularization term $\epsilon(t)x(t)$ is removed from the dynamical system \eqref{equ2_1}, such convergence fails to occur, as illustrated in \cref{fig: 1_convergence}(b). Similarly, for the dynamical system (PD-AVD) without the Tikhonov regularization term, the generated trajectory also fails to converge to the minimum-norm solution of the problem \eqref{equ52}. This highlights the crucial role played by the Tikhonov regularization in ensuring strong convergence to the minimum-norm solution.

\begin{figure}[htbp]
	\centering
	\subfigure[dynamical system (PDDS-IHDTR)]
	{
		\begin{minipage}[t]{0.32\linewidth}
			\centering
			\includegraphics[width=2in]{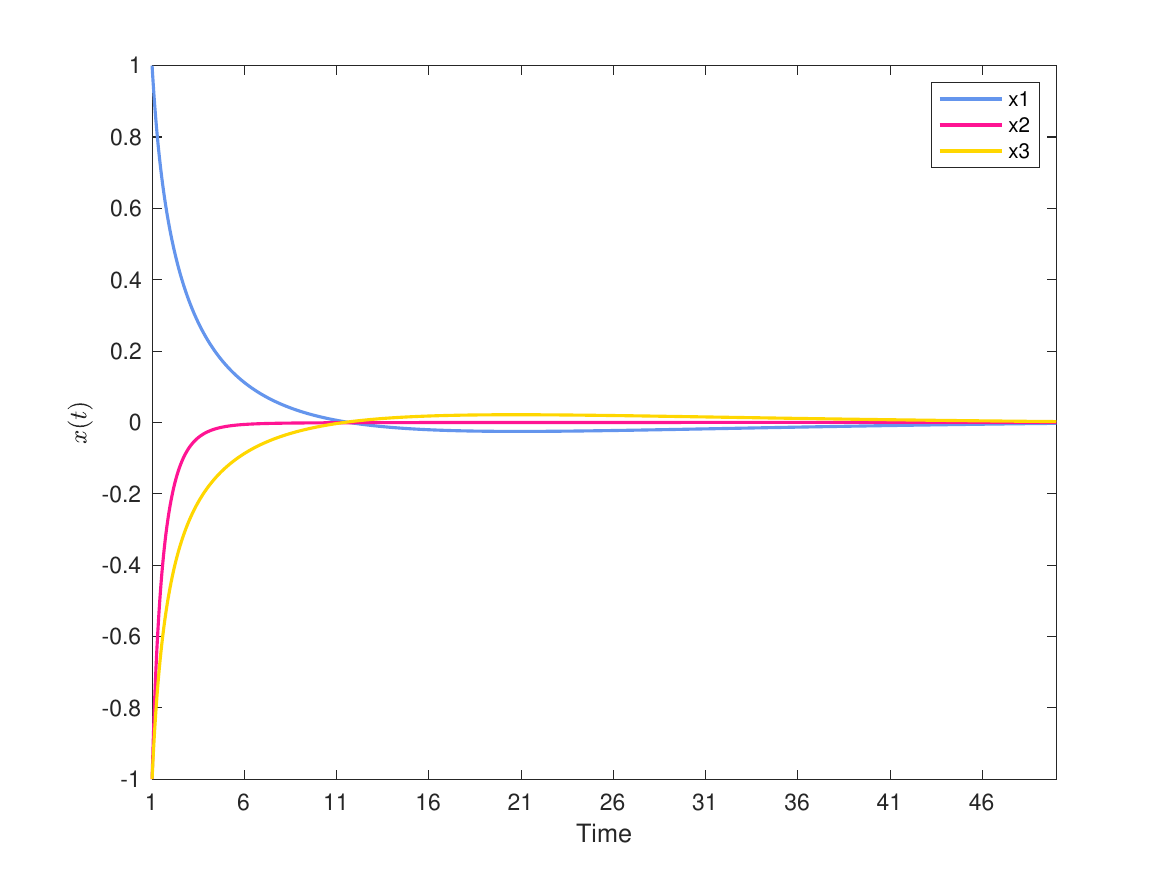}
		\end{minipage}%
	}
	\subfigure[dynamical system (PDDS-IHD)]
	{
		\begin{minipage}[t]{0.32\linewidth}
			\centering
			\includegraphics[width=2in]{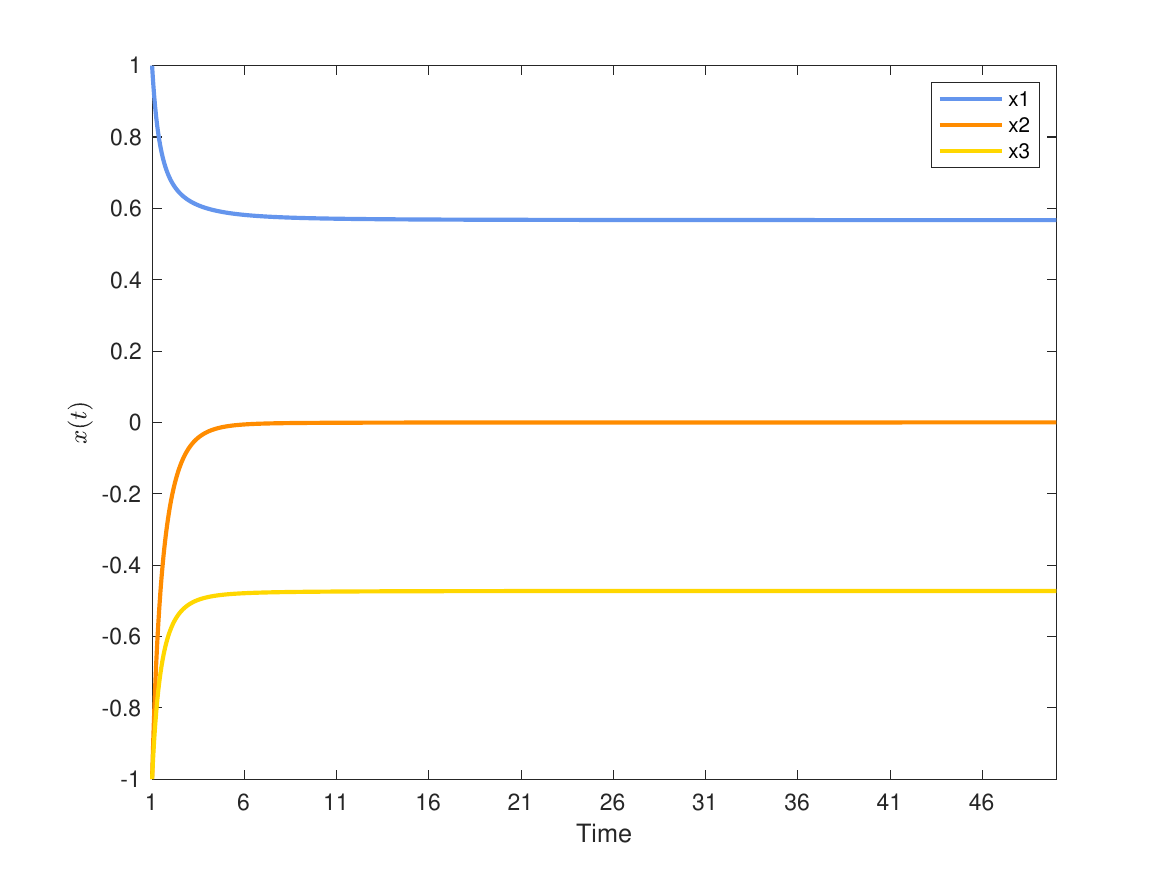}
		\end{minipage}%
	}
	\subfigure[dynamical system (PD-AVD)]
	{
		\begin{minipage}[t]{0.32\linewidth}
			\centering
			\includegraphics[width=2in]{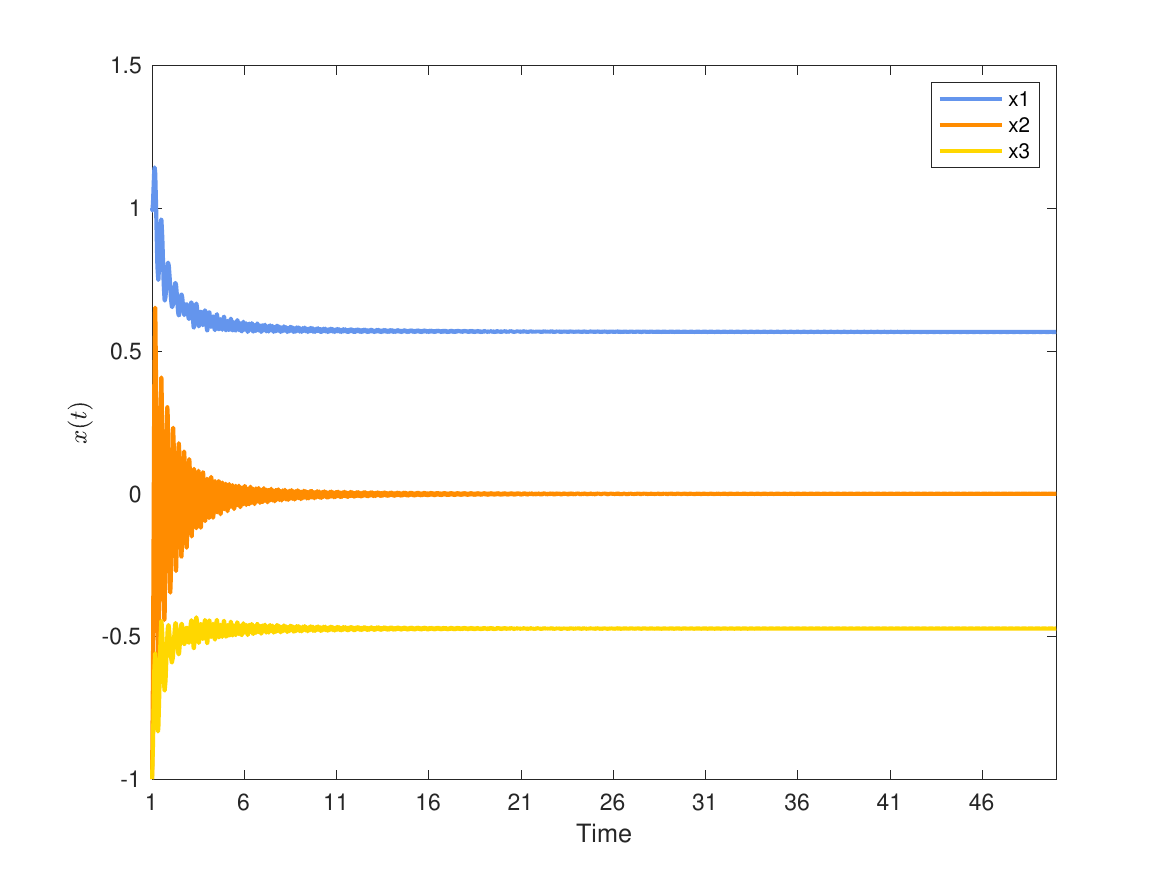}
		\end{minipage}%
	}
	\caption{The behaviors of the trajectories $x(t)$ generated by (PDDS-IHDTR), (PDDS-IHD) and (PD-AVD) for the problem \eqref{equ52}}
	\label{fig: 1_convergence}
	\centering
\end{figure}
\vspace{-5mm}
In the second experiment, we conduct a comparative study among five dynamical systems: (i) the  dynamical system (PDDS-IHDTR) \eqref{equ2_1}, (ii) the dynamical system \eqref{equ2_1} without the term $(\gamma+\frac{\beta}{t})\dot{x}(t)$, (iii) the dynamical system (PDDS-IHD) \eqref{equ2}, (iv) the dynamical system \eqref{equ2} without the term $(\gamma+\frac{\beta}{t})\dot{x}(t)$, and (v) the dynamical system (PD-AVD). For the dynamical system \eqref{equ2_1}, we consider two parameter configurations. In the first case, we set $m = 5$, $n = 10$, $e = 6$, $\alpha = 3.1$, $\gamma=1$, $\beta =-0.5$, $\rho=1$, $\xi(t) = 1$, $\epsilon(t) = \frac{1}{t^r}$ with $r=1.5$. In the second case, corresponding to the dynamical system without the term $(\gamma+\frac{\beta}{t})\dot{x}(t)$,  all parameters remain the same except that $\gamma $ and $\beta$ are omitted from the model. For the dynamical system (PDDS-IHD) \eqref{equ2}, we consider two parameter configurations, each corresponding to one of the two configurations used for the dynamical system (PDDS-IHDTR) \eqref{equ2_1}, with the sole difference that $\epsilon$ is removed from the model. For the (PD-AVD), we take $m = 5$, $n = 10$, $e = 6$, $\alpha = 3.1$, $\beta=1$, $\delta=\frac{1}{2.05}$. The initial conditions for the five dynamical system are given by $ (x(1),\lambda(1), \dot{x}(1), \dot{\lambda}(1))^T=\mathbf{1}^{8\times 1}$.

As shown in \cref{fig: 2_contrast}(a),  \cref{fig: 2_contrast}(b), and  \cref{fig: 2_contrast}(c), which depict the evolution of the primal-dual gap, feasibility violation, and iteration error along the trajectories generated by the five dynamical systems, respectively, the  perturbation term $(\gamma+\frac{\beta}{t})\dot{x}(t)$ in the argument of the gradient of the augmented Lagrange function $\mathcal{L}_{\rho}$ eliminates the oscillations. Specifically, the inclusion of this term results in smoother convergence trajectory and a notable reduction in fluctuations. Moreover, with respect to the primal-dual gap and the constraint violation, \cref{fig: 2_contrast}(a) and \cref{fig: 2_contrast}(b) demonstrate that dynamical systems (PDDS-IHDTR) \eqref{equ2_1} and (PDDS-IHD) \eqref{equ2} display the same decay behavior. Compared to the dynamical system  (PD-AVD), which does not include the perturbation term $(\gamma+\frac{\beta}{t})\dot{x}(t)$  and the Tikhonov regularization term $\epsilon(t)x(t)$, our proposed dynamical system \eqref{equ2_1} not only eliminates oscillations but also achieves better convergence accuracy, as shown in \cref{fig: 2_contrast}(c). The improved convergence accuracy is closely related to the ability of the Tikhonov regularization term $\epsilon(t)x(t)$ to ensure strong convergence of the trajectories.

\begin{figure}[htbp]
	\centering
	\subfigure[primal-dual gap]
	{
		\begin{minipage}[t]{0.32\linewidth}
			\centering
			\includegraphics[width=2in]{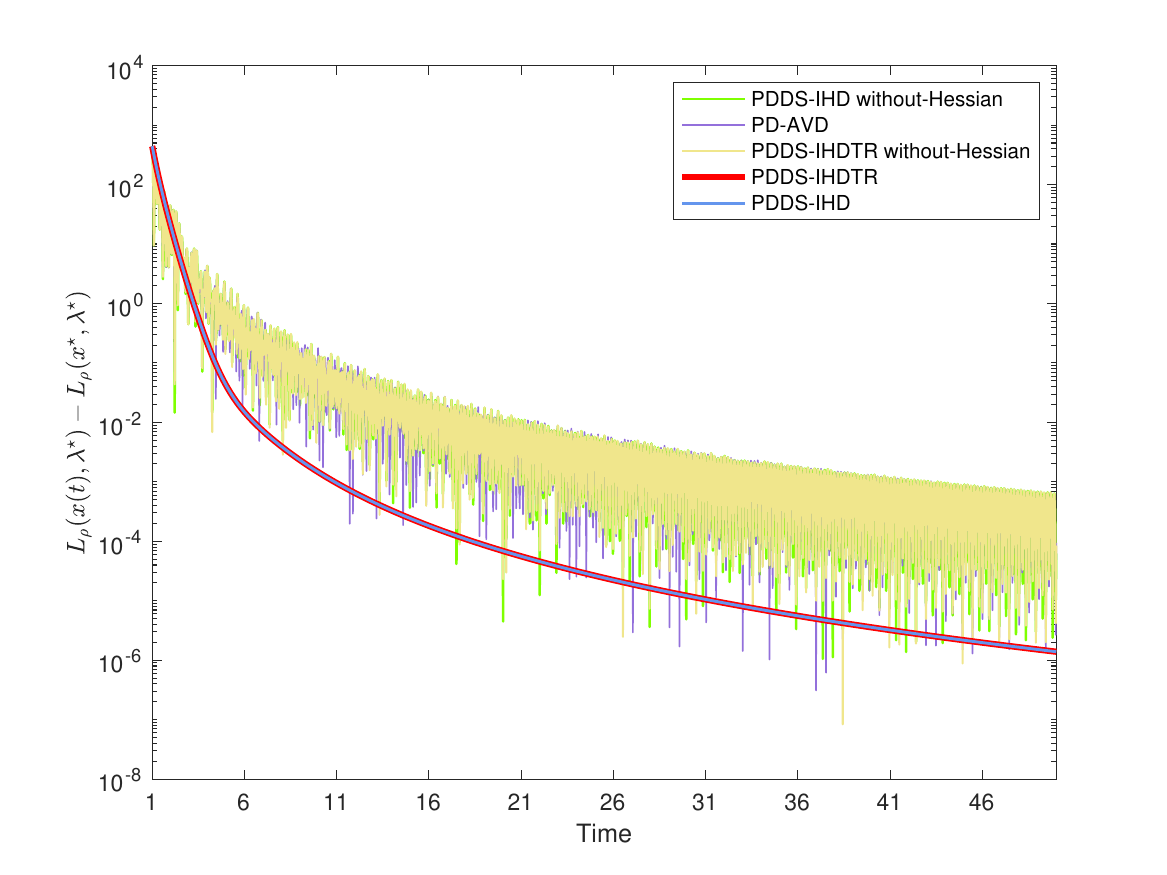}
		\end{minipage}
	}
	\subfigure[feasibility violation]
	{
		\begin{minipage}[t]{0.32\linewidth}
			\centering
			\includegraphics[width=2in]{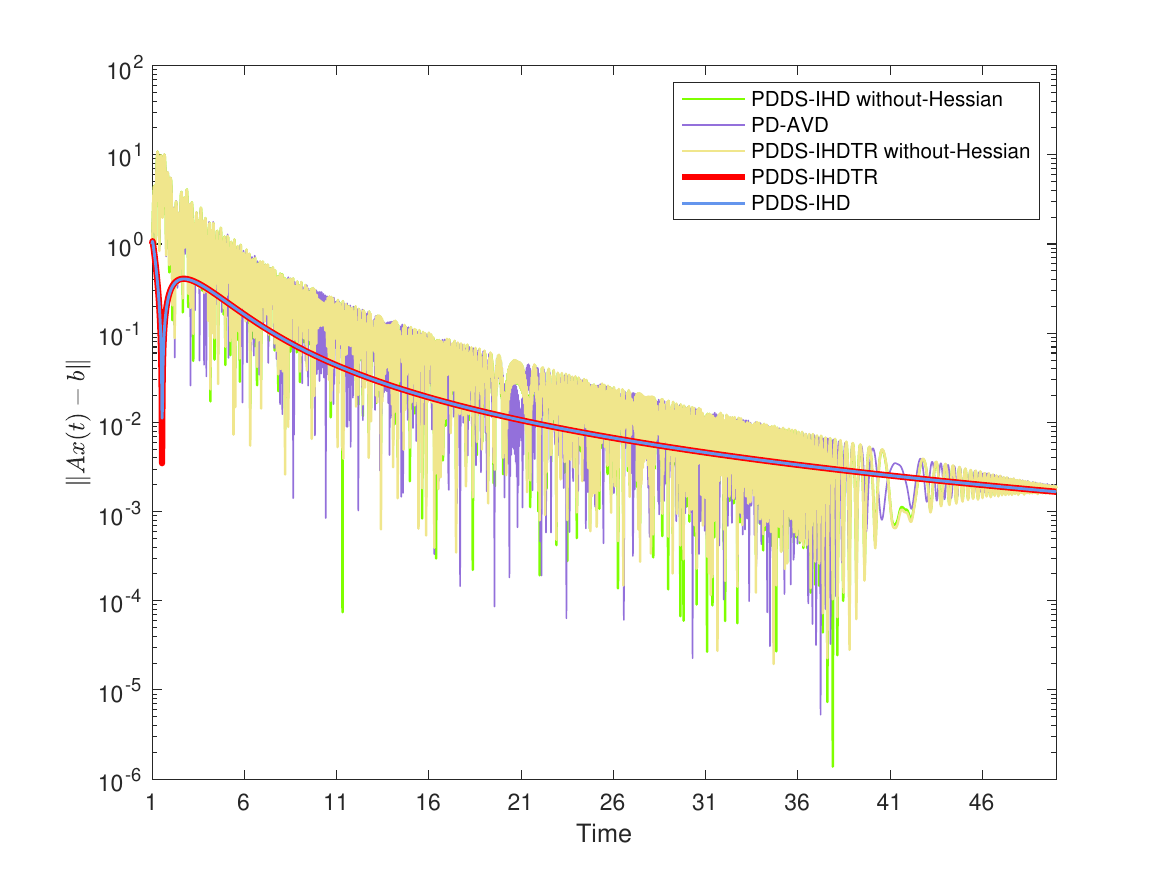}
		\end{minipage}%
	}
	\subfigure[iterate error]
	{
		\begin{minipage}[t]{0.32\linewidth}
			\centering
			\includegraphics[width=2in]{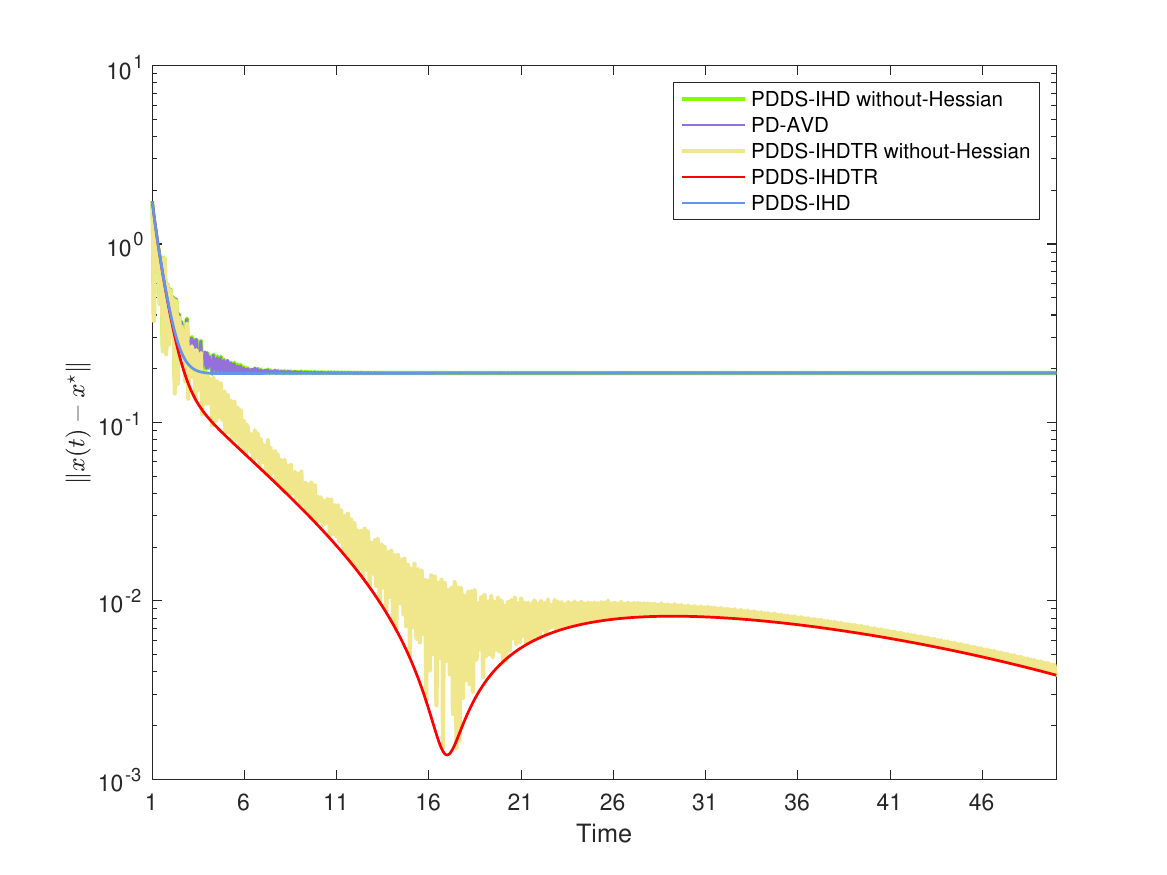}
		\end{minipage}%
	}
	\caption{Comparison of error results between (PD-AVD), (PDDS-IHDTR) without Hessian, (PDDS-IHDTR), (PDDS-IHD) without Hessian and  (PDDS-IHD)}
	\label{fig: 2_contrast}
	\centering
\end{figure}
Next, we conduct a third numerical experiment. In this experiment, we examine how the iteration error, constraint violation, and primal-dual gap evolve along the trajectory generated by the dynamical system \eqref{equ2_1} under varying values of $\gamma$ and $\beta$. The system is configured with the following parameter values: $m = 0.1$, $n = 20$, $e = 50$ $\alpha = 3.1$, $\rho=1$, $\xi(t) = 1$, $\epsilon(t) = \frac{1}{t^r}$, $r=1.1$, and initial conditions: $x(1)= (1, -1, -1)^T$, $\lambda(1) = 1$, $\dot{x}(1) = (-1, 1, 1)^T$, $\dot{\lambda}(1) = 1$.

As shown in \cref{fig: 3_contrast}, in terms of the initial decay rates of the three types of errors, the best parameter choice appears to be $\gamma=0$, $\beta>0$. However, in terms of the final accuracy, the optimal parameters seem to be $\gamma>0$, $\beta<0$. Moreover, the numerical results for this specific problem clearly indicate that the case $\gamma=0$, $\beta=0$, which corresponds to the absence of the Tikhonov regularization term $\epsilon(t)x(t)$, yields the worst performance among all parameter choices. This suggests that the combination of a positive $\gamma$ and a negative $\beta$ plays a crucial role in enhancing the asymptotic convergence behavior. On the other hand, while a positive $\beta$ alone may accelerate the early-stage convergence, it is insufficient for achieving high-precision solutions in the long run.
\begin{figure}[htbp]
	\centering
	\subfigure[iterate error]
	{
		\begin{minipage}[t]{0.32\linewidth}
			\centering
			\includegraphics[width=2in]{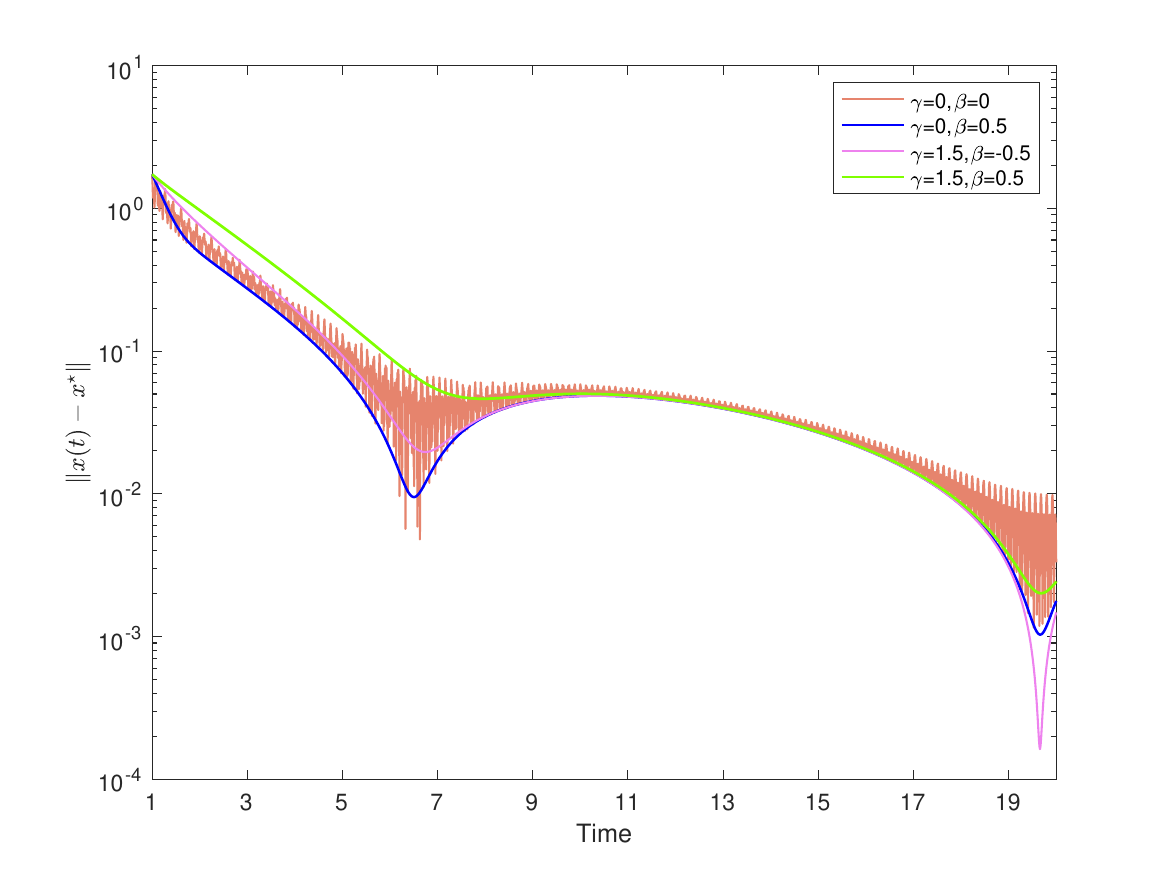}
		\end{minipage}
	}
	\subfigure[feasibility violation]
	{
		\begin{minipage}[t]{0.32\linewidth}
			\centering
			\includegraphics[width=2in]{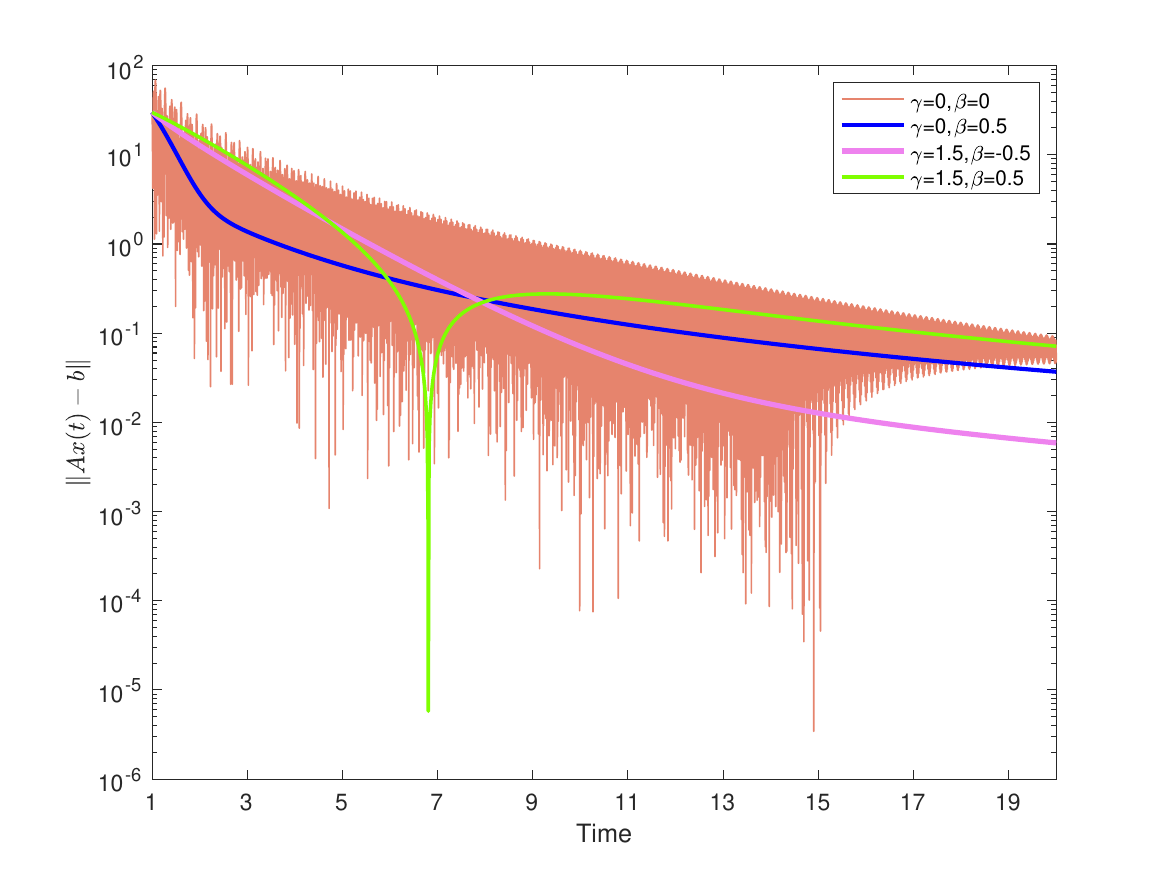}
		\end{minipage}%
	}
	\subfigure[primal-dual gap]
	{
		\begin{minipage}[t]{0.32\linewidth}
			\centering
			\includegraphics[width=2in]{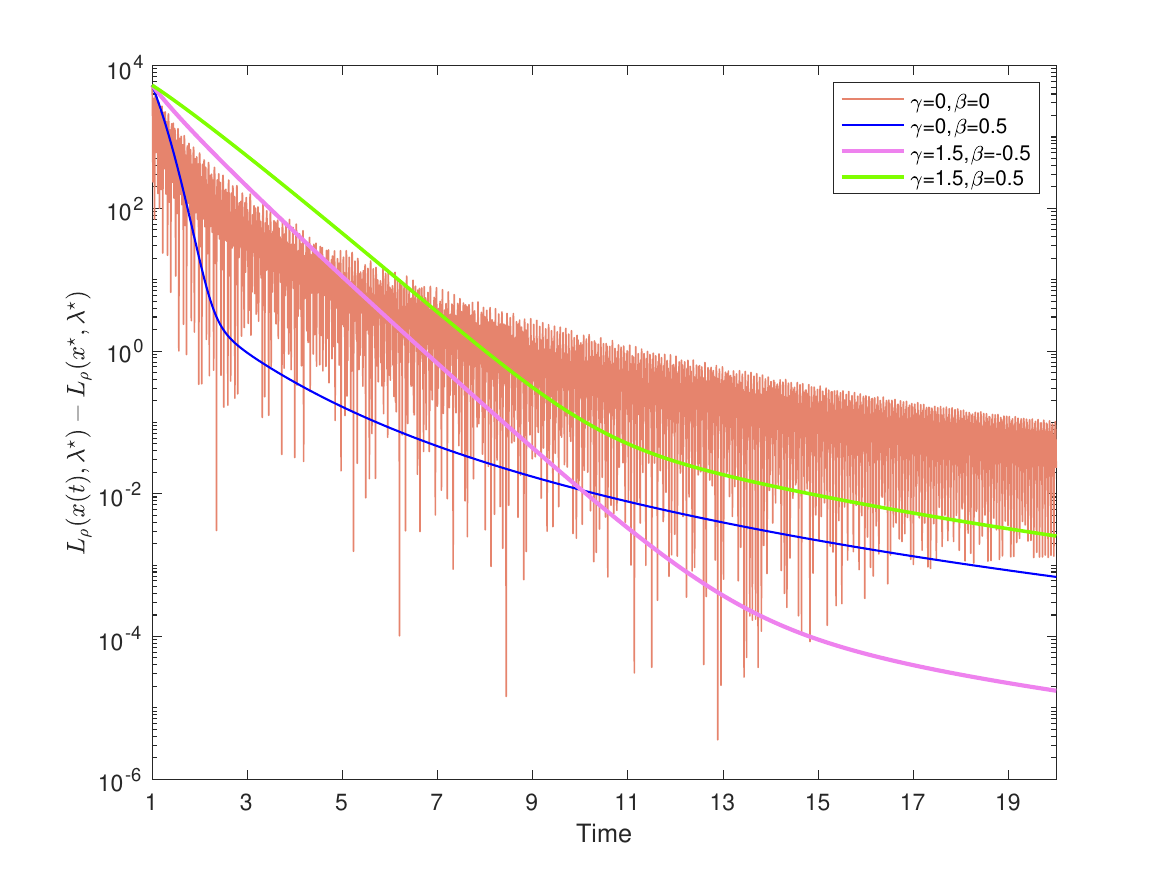}
		\end{minipage}%
	}
	\caption{Error analysis with different parameters in the dynamical system (PDDS-IHDTR) for the problem \eqref{equ52}}
	\label{fig: 3_contrast}
	\centering
\end{figure}

\section{Conclusion}\label{sec6}
In this paper, we propose two primal-dual dynamical systems \eqref{equ2} and \eqref{equ2_1}, which achieve smooth effects comparable to those of systems involving explicit Hessian-driven damping terms $\nabla^2f(x(t))\dot{x}(t)$, while only requiring the objective function to be continuously differentiable. 
Furthermore,  for both dynamical systems, we establish the convergence rate $\mathcal{O}(\frac{1}{t^2\xi(t)})$ for the error $\mathcal{L}_{\rho}(x(t)+(\gamma+\frac{\beta}{t})\dot{x}(t),\lambda(t))-\mathcal{L}_{\rho}(x^*,\lambda^*)$ and the convergence rate $\mathcal{O}(\frac{1}{t})$ for the velocity $\dot{x}(t)$ under the sole assumption that the objective function is convex.
Furthermore, by incorporating a Tikhonov regularization term into the dynamical system \eqref{equ2}, we introduce the dynamical system \eqref{equ2_1} and prove that its trajectory converges strongly to the minimum-norm solution of the problem \eqref{equ1}. Finally, we present numerical experiments to illustrate the theoretical findings.

	\appendix

\section*{CRediT authorship contribution statement}
{\bf Hong-lu Li:} Conceptualization, Software, Visualization, Writing-original draft.
{\bf Xin He:} Methodology, Investigation,  Funding acquisition, Writing-review \& editing. 
{\bf Yi-bin Xiao:} Conceptualization, Supervision, Funding acquisition, Writing-review \& editing.

\section*{Declaration of competing interest}

The authors declare that they have no known competing financial interests or personal relationships that could have appeared
to influence the work reported in this paper.

\section*{Data availability}
No data was used for the research described in the article.

\section*{Acknowledgments}
This research was supported by the National Natural Science Foundation of China (12171070) and Sichuan Science and Technology Program (Grant No. 2025ZNSFSC0813).


\begin{appendices}
	
	\section{Well-posedness of the Cauchy problem}\label{Appendix A}
	
	In this appendix, we will give the proof of the existence and uniqueness of a global solution of dynamical system \eqref{equ2} and \eqref{equ2_1} with the initial condition $(x(t_0),\lambda(t_0))=(x_0,\lambda_0)$, $(\dot{x}(t_0), \dot{\lambda}(t_0))=(u_0, w_0)$. Note that when $\epsilon =0$, the dynamical system \eqref{equ2_1} reduces to the dynamical system \eqref{equ2}. Therefore, in what follows, we focus on proving the existence and uniqueness of a global solution for the dynamical system \eqref{equ2_1}. The corresponding result for the dynamical system \eqref{equ2} will then follow by setting $\epsilon =0$ in the analysis.
	\vskip 2mm
	\begin{definition}
		A function $x : [t_0, +\infty)\rightarrow\mathcal{X}$ is a strong global solution of the dynamical system \eqref{equ2_1}, if it satisfies the following properties:
		\begin{enumerate}
			\item[(i)] $x : [t_0, +\infty)\rightarrow\mathcal{X}$ and $\lambda : [t_0, +\infty)\rightarrow\mathcal{Y}$ are locally absolutely continuous;
			\item[(ii)] 		
			$
			\begin{cases}
				\begin{aligned}
					&\ddot{x}(t) + \frac{\alpha}{t} \dot{x}(t) + \xi(t) (\nabla_x\mathcal{L}_{\rho}(x(t)+\beta(t)\dot{x}(t),\lambda(t)+\frac{3}{2\alpha} t\dot{\lambda}(t)) + \epsilon(t) x(t)) = 0, \\
					&\ddot{\lambda}(t) + \frac{\alpha}{t} \dot{\lambda}(t) - \xi(t) (\nabla_{\lambda}\mathcal{L}_{\rho}(x(t)+\frac{3}{2\alpha} t\dot{x}(t),\lambda(t))+\frac{1}{2\alpha}(-\alpha\beta(t)+3 t\dot{\beta}(t))A\dot{x}(t)\\
					&\qquad-\frac{3}{2\alpha} t\xi(t)\beta(t)(A\nabla_x\mathcal{L}_{\rho}(x(t)+\beta(t)\dot{x}(t), \lambda(t)+\frac{3}{2\alpha} t\dot{\lambda}(t))+\epsilon(t)Ax(t)))= 0,
				\end{aligned} \quad\forall t\in[t_0, +\infty)
			\end{cases} 
			$
			\item[(iii)] $(x(t_0), \lambda(t_0))=(x_0, \lambda_0)\in \mathcal{X}\times\mathcal{Y}$, $(\dot{x}(t_0), \dot{\lambda}(t_0))=(u_0,w_0)\in \mathcal{X}\times\mathcal{Y}$.
		\end{enumerate}
		\vskip 2mm
	\end{definition}
	\begin{theorem}\label{thm1}
		Suppose that $f$ is a continuous function such that $\nabla f$ is Lipschitz continuous with Lipschitz constant $L_f>0$. Let $\epsilon, \xi:\left[t_0,+\infty\right)\rightarrow\left(0,+\infty\right)$ be continuous functions. Then, for any initial condition $(x(t_0), \lambda(t_0))=(x_0, \lambda_0)\in \mathcal{X}\times\mathcal{Y}$, $(\dot{x}(t_0), \dot{\lambda}(t_0))=(u_0, w_0)\in \mathcal{X}\times\mathcal{Y}$ with $t_0>0$, the dynamical system \eqref{equ2_1} has a unique strong global solution.
	\end{theorem}
	\begin{proof}
		Denote $Z(t)$ := $(z_1(t),z_2(t),z_3(t),z_4(t))$ = $(x(t), \lambda(t),\dot{x}(t),\dot{\lambda}(t))$, $Z_0$ := $(x_0,\lambda_0,u_0,w_0)$. For the ease of writing, we denote $Z(t)$ as $Z$ := $(z_1,z_2,z_3,z_4)$. Then, the dynamical system \eqref{equ2_1} can be rewritten as: 
		\begin{equation}\label{equA1}
			\begin{cases}
				\frac{dZ}{dt}=F(t,Z),\\
				Z(t_0)=Z_0,
			\end{cases}
		\end{equation}
		where
		\begin{equation*}
			F(t,Z) := \left(
			\begin{aligned}
				&z_3\\
				&z_4\\
				& -\frac{\alpha}{t}z_3-\xi(t)(\nabla_x\mathcal{L}_{\rho}(z_1+\beta(t)z_3, z_2+\frac{3}{2\alpha} tz_4) + \epsilon(t) z_1)\\
				&-\frac{\alpha}{t}z_4+ \xi(t) (\nabla_{\lambda}\mathcal{L}_{\rho}(z_1+\frac{3}{2\alpha} tz_3, z_2)
				+\frac{1}{2\alpha}(-\alpha\beta(t)+3 t\dot{\beta}(t))(Az_4)\\
				&\qquad\quad-\frac{3}{2\alpha} t\xi(t)\beta(t)(A\nabla_x\mathcal{L}_{\rho}(z_1+\beta(t)z_3, z_2+\frac{3}{2\alpha} tz_4)
				+\epsilon(t)Az_1))
			\end{aligned}\right).
		\end{equation*}
		Since $\nabla f$ is Lipschitz continuous and $A$ is linear, it follows that for any $Z$, $\widetilde{Z}\in \mathcal{X}\times\mathcal{Y}\times\mathcal{X}\times\mathcal{Y}$,
		\begin{eqnarray*}
			\Vert F(s,Z(s))-F(s,\widetilde{Z}(s))\Vert&\leq&(M(t):=1+\frac{\alpha}{t}+\xi(t)((1+\frac{3\Vert A\Vert}{2\alpha}t\xi(t))((1+\gamma+\frac{\beta}{t})(L_f+\rho\Vert A\Vert^2)+\epsilon(t))\\
			&&+(1+\frac{1}{2\alpha}(3t-\alpha\gamma-\frac{(\alpha+3)\beta}{t}))\Vert A\Vert\\
			&&+\frac{3}{2\alpha}(1+\frac{3}{2\alpha}t)(\gamma+\frac{\beta}{t})t\xi(t)\Vert A\Vert^2
			))\Vert Z-\widetilde{Z}\Vert.
		\end{eqnarray*}
		Since $\epsilon,\xi:\left[t_0,+\infty\right)\rightarrow\left(0,+\infty\right)$ are continuous functions, we have $M(t)\in L_{loc}^{1}\left[t_0,+\infty\right)$. Further, by the Lipschitz continuity of 
		$\nabla f$ , we deduce that
		\begin{equation*}
			\Vert \nabla f(z_1+\beta(t)z_4)\Vert\leq\Vert\nabla f(0)\Vert+ L_f\Vert z_1+\beta(t)z_4\Vert.
		\end{equation*}
		Therefore, for any given $Z\in\mathcal{X}\times\mathcal{Y}\times\mathcal{X}\times\mathcal{Y}$ and $t_0< T <+\infty$, the following inequality holds:
		\begin{eqnarray*}
			\int_{t_0}^{T}\Vert F(t,Z)\Vert dt&\leq& \int_{t_0}^{T}( S(t):=M(t)+(1+\frac{3\Vert A\Vert}{2\alpha}(\gamma+\frac{\beta}{t})t\xi(t))\Vert \nabla f(0)\Vert  \\
			&&+(1+\rho\Vert A\Vert +\frac{3\rho\Vert A\Vert^2}{2\alpha}(\gamma+\frac{\beta}{t})t\xi(t))\vert c\vert
			)(1+\Vert Z\Vert)dt.
		\end{eqnarray*}
		It is clear that $S(t)\in L_{loc}^{1}\left[t_0,+\infty\right)$ and for any $Z\in\mathcal{X}\times\mathcal{Y}\times\mathcal{X}\times\mathcal{Y}$, $F(\cdot, Z)\in L_{loc}^{1}([t_0,+\infty); \mathcal{X}\times\mathcal{Y}\times\mathcal{X}\times\mathcal{Y})$. According to \cite[Proposition 6.2.1]{bib24} and \cite[Theorem 5]{bib15}, the  Cauchy problem \eqref{equA1} has a unique global solution $Z\in W_{loc}^{1,1}([t_0,+\infty);\mathcal{X}\times\mathcal{Y}\times\mathcal{X}\times\mathcal{Y})$. Therefore, the dynamical system \eqref{equ2_1} has a unique strong global solution $(x, \lambda)$. This completes the proof of Theorem \ref{thm1}.
	\end{proof}

	\section{Some auxiliary results}\label{secB}

	\begin{lemma}{\rm{\cite[Lemma A.3]{bib21}}}\label{lemB.2}
		Suppose that $\delta>0$ and $\phi\in L^1\left(\delta,+\infty\right)$ is a non-negative and continuous function, and $\psi:\left[\delta,+\infty\right)\rightarrow\left(0,+\infty\right)$ is a non-decreasing function such that $\lim\limits_{t\rightarrow+\infty}\psi(t)=+\infty$. Then,
		\begin{equation*}
			\lim\limits_{t\rightarrow+\infty}\frac{1}{\psi(t)}\int_{\delta}^{t}\psi(s)\phi(s)ds =0.
		\end{equation*}
	\end{lemma}
	
	
	
	
\end{appendices}


\end{document}